\documentclass[11pt]{amsart}

\usepackage{graphics}
\usepackage{color}
\usepackage[a4paper, margin=3cm]{geometry}

\usepackage{amssymb,enumerate}
\usepackage{amsmath}
\usepackage{bbm}           
\usepackage{bm}
\usepackage{eso-pic,graphicx}
\usepackage{tikz}
\usepackage{cite}
\usepackage{esint}
\usepackage[colorlinks=true, pdfstartview=FitV, linkcolor=blue, citecolor=blue, urlcolor=blue]{hyperref}

\usepackage{graphicx}
\usepackage{pslatex}
\usepackage{amsmath,amsfonts}
\usepackage{rotating}
\usepackage{amssymb}
\usepackage{verbatim}
\usepackage{rotating}
\usepackage{mathrsfs}

\makeatletter
\def\Ddots{\mathinner{\mkern1mu\raise\p@
\vbox{\kern7\p@\hbox{.}}\mkern2mu
\raise4\p@\hbox{.}\mkern2mu\raise7\p@\hbox{.}\mkern1mu}}
\makeatother

\def\XXint#1#2#3{{\setbox0=\hbox{$#1{#2#3}{\int}$}
\vcenter{\hbox{$#2#3$}}\kern-.5\wd0}}

\begin{document}
\newtheorem{theorem}{Theorem}
\newtheorem{proposition}[theorem]{Proposition}
\newtheorem{conjecture}[theorem]{Conjecture}
\def\theconjecture{\unskip}
\newtheorem{corollary}[theorem]{Corollary}
\newtheorem{lemma}[theorem]{Lemma}
\newtheorem{claim}[theorem]{Claim}
\newtheorem{sublemma}[theorem]{Sublemma}
\newtheorem{observation}[theorem]{Observation}

\newtheorem{definition}[theorem]{Definition}

\newtheorem{remark}[theorem]{Remark}

\def\C{\mathbb{C}}
\def\R{\mathbb{R}}
\def\Rn{{\mathbb{R}^n}}
\def\Rns{{\mathbb{R}^{n+1}}}
\def\Sn{{{S}^{n-1}}}
\def\M{\mathbb{M}}
\def\N{\mathbb{N}}
\def\Q{{\mathbb{Q}}}
\def\Z{\mathbb{Z}}
\def\F{\mathcal{F}}
\def\L{\mathcal{L}}
\def\S{\mathcal{S}}
\def\supp{\operatorname{supp}}
\def\essi{\operatornamewithlimits{ess\,inf}}
\def\esss{\operatornamewithlimits{ess\,sup}}

\def\earrow{{\mathbf e}}
\def\rarrow{{\mathbf r}}
\def\uarrow{{\mathbf u}}
\def\varrow{{\mathbf V}}
\def\tpar{T_{\rm par}}
\def\apar{A_{\rm par}}

\def\reals{{\mathbb R}}
\def\torus{{\mathbb T}}
\def\scriptm{{\mathcal T}}
\def\heis{{\mathbb H}}
\def\integers{{\mathbb Z}}
\def\z{{\mathbb Z}}
\def\naturals{{\mathbb N}}
\def\complex{{\mathbb C}\/}
\def\distance{\operatorname{distance}\,}
\def\support{\operatorname{support}\,}
\def\dist{\operatorname{dist}\,}
\def\Span{\operatorname{span}\,}
\def\degree{\operatorname{degree}\,}
\def\kernel{\operatorname{kernel}\,}
\def\dim{\operatorname{dim}\,}
\def\codim{\operatorname{codim}}
\def\trace{\operatorname{trace\,}}
\def\Span{\operatorname{span}\,}
\def\dimension{\operatorname{dimension}\,}
\def\codimension{\operatorname{codimension}\,}
\def\nullspace{\scriptk}
\def\kernel{\operatorname{Ker}}
\def\ZZ{ {\mathbb Z} }
\def\p{\partial}
\def\rp{{ ^{-1} }}
\def\Re{\operatorname{Re\,} }
\def\Im{\operatorname{Im\,} }
\def\ov{\overline}
\def\eps{\varepsilon}
\def\lt{L^2}
\def\diver{\operatorname{div}}
\def\curl{\operatorname{curl}}
\def\etta{\eta}
\newcommand{\norm}[1]{ \|  #1 \|}
\def\expect{\mathbb E}
\def\bull{$\bullet$\ }

\def\blue{\color{blue}}
\def\red{\color{red}}

\def\xone{x_1}
\def\xtwo{x_2}
\def\xq{x_2+x_1^2}
\newcommand{\abr}[1]{ \langle  #1 \rangle}

\newcommand{\Norm}[1]{ \left\|  #1 \right\| }
\newcommand{\set}[1]{ \left\{ #1 \right\} }
\newcommand{\ifou}{\raisebox{-1ex}{$\check{}$}}
\def\one{\mathbf 1}
\def\whole{\mathbf V}
\newcommand{\modulo}[2]{[#1]_{#2}}
\def \essinf{\mathop{\rm essinf}}
\def\scriptf{{\mathcal F}}
\def\scriptg{{\mathcal G}}
\def\scriptm{{\mathcal M}}
\def\scriptb{{\mathcal B}}
\def\scriptc{{\mathcal C}}
\def\scriptt{{\mathcal T}}
\def\scripti{{\mathcal I}}
\def\scripte{{\mathcal E}}
\def\scriptv{{\mathcal V}}
\def\scriptw{{\mathcal W}}
\def\scriptu{{\mathcal U}}
\def\scriptS{{\mathcal S}}
\def\scripta{{\mathcal A}}
\def\scriptr{{\mathcal R}}
\def\scripto{{\mathcal O}}
\def\scripth{{\mathcal H}}
\def\scriptd{{\mathcal D}}
\def\scriptl{{\mathcal L}}
\def\scriptn{{\mathcal N}}
\def\scriptp{{\mathcal P}}
\def\scriptk{{\mathcal K}}
\def\frakv{{\mathfrak V}}
\def\C{\mathbb{C}}
\def\D{\mathcal{D}}
\def\R{\mathbb{R}}
\def\Rn{{\mathbb{R}^n}}
\def\rn{{\mathbb{R}^n}}
\def\Rm{{\mathbb{R}^{2n}}}
\def\r2n{{\mathbb{R}^{2n}}}
\def\Sn{{{S}^{n-1}}}
\def\M{\mathbb{M}}
\def\N{\mathbb{N}}
\def\Q{{\mathcal{Q}}}
\def\Z{\mathbb{Z}}
\def\F{\mathcal{F}}
\def\L{\mathcal{L}}
\def\G{\mathscr{G}}
\def\ch{\operatorname{ch}}
\def\supp{\operatorname{supp}}
\def\dist{\operatorname{dist}}
\def\essi{\operatornamewithlimits{ess\,inf}}
\def\esss{\operatornamewithlimits{ess\,sup}}
\def\dis{\displaystyle}
\def\dsum{\displaystyle\sum}
\def\dint{\displaystyle\int}
\def\dfrac{\displaystyle\frac}
\def\dsup{\displaystyle\sup}
\def\dlim{\displaystyle\lim}
\def\bom{\Omega}
\def\om{\omega}
\author[F. Liu]{Feng Liu}
\address{Feng Liu:
        College of Mathematics and System Science\\
        Shandong University of Science and Technology\\
        Qingdao, Shandong 266590\\
        People's Republic of China}
\email{FLiu@sdust.edu.cn}

\author[S. Xi]{Shuai Xi}
\address{Shuai Xi:
        College of Mathematics and System Science\\
        Shandong University of Science and Technology\\
        Qingdao, Shandong 266590\\
        People's Republic of China}
\email{shuaixi@sdust.edu.cn}

\author[S. Zhu]{Shengguo Zhu}
\address{Shengguo Zhu: Mathematical Institute, University of Oxford, Oxford OX2 6GG, UK.}
\email{\tt zhus@maths.ox.ac.uk}


\date{\today}
\title[Global mild solutions to three-dimensional magnetohydrodynamic
system]
{Global mild solutions to three-dimensional magnetohydrodynamic
system in Morrey spaces}
\maketitle

\begin{abstract}
In this article, the Cauchy problem of three-dimensional (3-D)
incompressible magnetohydrodynamic system was investigated. If the
initial $\mathcal{M}^{1,1}$ norms of the vorticity $\omega$ and the current
density $j$ are both sufficiently small, then some uniform estimates
with respect to time for the coupling terms between the fluid and
the magnetic field can be established, which lead to a global-in-time
well-posedness of mild solutions in Morrey spaces via some effective
arguments.
\end{abstract}

\section{Introduction}
Magnetohydrodynamics (\textbf{MHD}) is that part of the mechanics of
continuous media which studies the motion of electrically conducting
media in the presence of a magnetic field. The dynamic motion of fluid
and magnetic field interact strongly on each other, so the hydrodynamic
and electrodynamic effects are coupled. The applications of
magnetohydrodynamics cover a very wide range of physical objects, from
liquid metals to cosmic plasmas, for example, the intensely heated and
ionized fluids in an electromagnetic field in astrophysics, geophysics,
high-speed aerodynamics, and plasma physics. The motion of an electrically
conducting, viscous incompressible fluid in $\mathbb{R}^3$ can be described
by the Magnetohydrodynamic equations (see \cite{Cow,LL}):
$$
\partial_tu-\nu\triangle u+(u\cdot\nabla)u-(b\cdot\nabla)b+\nabla\Big(P+\frac{1}{2}|b|^2\Big)=0\ \ \ {\rm in}\ \ \  \mathbb{R}^3\times(0,\infty),\qquad \eqno(1.1)$$
$$\qquad \qquad \quad \ \  \ \ \ \  \ \  \partial_tb-\eta\triangle b+\nabla\times(b\times u)=0\ \ \ \ {\rm in}\ \ \ \mathbb{R}^3\times(0,\infty),\eqno(1.2)$$
$$\qquad \qquad \quad \qquad \qquad \qquad    \ \ \ \  \ \ {\rm div} u={\rm div}b=0\ \ \   {\rm in}\ \ \  \mathbb{R}^3\times(0,\infty), \eqno(1.3)$$
and the initial data
$$u(x,0)=u_0(x),\quad b(x,0)=b_0(x)\quad \text{for} \quad  x\in\mathbb{R}^3,\eqno(1.4)$$
where $(x=(x_1,x_2,x_3)^\top,\ t)\in \mathbb{R}^3\times \mathbb{R}_+$,
the unknowns $u(x,t)=(u^1,u^2,u^3)^\top(x,t)$, $P(x,t)$ and
$b(x,t)=(b^1,b^2,b^3)^\top(x,t)$ denote the fluid velocity, pressure
and magnetic field, respectively. Here, $P_T=P+\frac{1}{2}|b|^2$ is
the total kinetic pressure, $\nu$ is the kinematic viscosity and
$\eta$ is the resistivity. For simplicity, we assume throughout that
both $\nu$ and $\eta$ are equal to 1.

The global existence of \textbf{MHD} equations with finite energy initial data, i.e.,
$(u_0,b_0)\in L^2(\mathbb{R}^3)$ was independently proved by
Duvaut-Lions \cite{DL} and Sermange-Temam \cite{ST}. Mild solutions,
for the Navier-Stokes ({\bf NS}) equations, were first constructed
by Kato-Fujita \cite{fujk64} in the spaces $H^s(\mathbb{R}^d)$ for
$s\geq \frac{d}{2}-1$, and then in $L^p(\mathbb{R}^d)$ ($p\geq d$)
spaces (classical admissible spaces) by Kato \cite{Ka1}. The other
global well-posedness of mild solutions for small initial data is
due to Cannone \cite{c97} and Planchon \cite{p96} in the Besov space
$B^{-1+\frac{d}{p}}_{p,\infty}(\mathbb{R}^d)$, with $1<p<\infty$,
Koch-Tataru \cite{koct01} in the space $BMO^{-1}$, and Lei-Lin
\cite{ll11} in the space
$\chi^{-1}=\left\{f\in\mathcal{D}'(\mathbb{R}^3):\int_{\mathbb{R}^3}
|\xi|^{-1}|\hat{f}|(\xi){\rm d}\xi<\infty\right\}$.

Our purpose, in this paper, is to establish well-posedness results
of {\bf MHD} equations with the initial data highly concentrated
in ``small sets''(``rough data''), such as the initial vortex
profiles are vortex rings and filaments. Such kind of rough initial
data is of infinite energy and thus outside the scope of the
standard energy method and the classical Leray's theory (\cite{DL},
\cite{ST}).

For the case of the multi-dimensional incompressible Navier-Stokes
({\bf NS}) equations, the well-posedness with such kind of rough
initial data has been studied by many mathematicians in this field.
We attribute the corresponding development into two stages:

{\bf Stage I.} (Two-dimensional (2-D) {\bf NS} equation). The global
existence of solutions to 2-D {\bf NS} equations with large initial
data in the space of finite measures $\mathcal{M}(\mathbb{R}^2)$ was
first shown by Cottet \cite{Co}, and independently by
Giga-Miyakawa-Osada \cite{GMO}, and the proof in \cite{GMO} was
simplified by Kato \cite{Ka3}. Under the assumption that the atomic
part of the initial vorticity is sufficiently small, uniqueness was
also proved in \cite{GMO,Ka3}. Later, Gallagher-Gallay \cite{GG}
extended the uniqueness theory to the general initial data in
$\mathcal{M}(\mathbb{R}^2)$ via some arguments developed in \cite{GW},
which allows to handle large Dirac masses. In summary, {\bf NS}
equations are well-posed for arbitrary data in the space of
$\mathcal{M}(\mathbb{R}^2)$.

{\bf Stage II.} (3-D {\bf NS} equation). For 3-D {\bf NS} equations,
to the best of our knowledge, the progress is not as good as that of
2-D case, of which the reason is mainly due to the different forms
of vorticity stretching term in the  equation of vorticity, which
can be shown as follows,
$$(u\cdot\nabla)\omega \mbox{\quad in } \ \mathbb{R}^2 \quad \text{and} \quad   (u\cdot\nabla)\omega-(\omega\cdot\nabla)u \mbox{\quad in } \ \mathbb{R}^3.$$
As a result, compared with the 2-D case, more estimates on $u$ are
required to control the $ L^1(\mathbb{R}^3)$ norm of the vorticity.
The first attempt in the measure direction, for the well-posedenss
of the 3-D {\bf NS} equations, was done by Cottet-Soler \cite{CS2},
and they obtained the global existence and uniqueness of the solution
when the initial data is a linear combination of filament measures,
which is sufficiently small. Here, a filament measure is a vector
valued measure supported on a finite disjoint union of uniformly $C^1$
but non-closed curves. This implies that the total variation of the
initial data in \cite{CS2} is  finite.

In order to get rid of the finite total variation assumption, Giga
and Miyakawa introduced the Morrey-type space
$\mathcal{M}^{p}(\mathbb{R}^3)$ of measure (see {\bf Definition 2.1})
in \cite{GM2}, and established the existence of global solutions whose
initial vorticity is a small Radon measure belonging to
${\mathcal{M}}^\frac{3}{2}(\mathbb{R}^3)$. In order to solve {\bf NS}
equations for the velocity $u(x,t)$ instead of the vorticity $\omega$,
Kato \cite{Ka2} introduced a more general Morrey space $\mathcal{M}^{m,
\lambda}(\mathbb{R}^3)$ (see {\bf Definition 2.1}) and established a
more general result: if $u_0$ is small enough in
$\mathcal{M}^{m,3-m}(\mathbb{R}^3)$ for $1<m\leq 3$, then there is a
global solution $u\in\mathcal{M}^{m,3-m}(\mathbb{R}^3)$, which is
unique subject to certain restrictions. This partially generalizes the
result of \cite{GM2} by eliminating the differentiability of $u_0$
except the case for  $m=1$. The analysis on the meaningful case, $m=1$,
which admits certain measures, encounters some essential mathematical
difficulties, since the Calderon-Zygmund type singular operator is not
bound in $L^1(\mathbb{R}^3)$. Some other related  progress, for 3-D
{\bf NS} equations in Morrey  space, can also be found in
\cite{ fed93,{Tay},{can95},{kozy97}}.

For the {\bf MHD} system, the situation is more complicated due to
the strong coupling effect between the velocity vector field $u(x,t)$
and the magnetic field $b(x,t)$. In order to clearly analyse the
coupling effect of this system, we give the equations on the time
evolution of the vorticity $\omega=(\omega^1,\omega^2,\omega^3)^\top$
and the current density $j=(j^1,j^2,j^3)^\top$:
$$\begin{array}{ll}
\left\{\aligned
&\partial_t\omega-\triangle \omega+\partial_{x_i}( u^i\omega-\omega^iu-b^ij+j^ib)=0,\\[4pt]
&\partial_t j-\triangle j+\nabla\times((u\cdot\nabla)b-(b\cdot\nabla)u)=0,\\[4pt]
&u=K\ast\omega,\quad b=K\ast j,\\[4pt]
&{\rm div}\omega={\rm div}j=0,\\[4pt]
&\omega(x,0)=\omega_0,\quad j(x,0)=j_0,
\endaligned\right.
\end{array}\eqno(1.5)$$
where $K(x)=-\frac{1}{4\pi}\frac{(x_1,x_2,x_3)}{|x|^3}$ is the
Biot-Savart kernel. Based some elaborate analysis of  the above system,
in the current paper, we will present the global mild solutions of (1.5).
If the $W^{1,1}$ norms of $\omega$ and $j$ are sufficiently small
initially, then some uniform estimates with respect to time for the
coupling terms between the fluid and the magnetic field can be obtained,
which lead to a global-in-time well-posedness of mild solutions in
Morrey spaces via some effective arguments.

\bigskip

\section{Preliminary definitions and main results}\label{S2}
In this section, we will give some necessary definitions and then state
our main result. Moreover, the main strategy of our following proof will
be also discussed.
\bigskip

\quad\hspace{-20pt}{\bf 2.1. Preliminary definitions}

\medskip
We start with the definition of Morrey spaces.

\medskip

\quad\hspace{-20pt}{\bf Definition 2.1} (\cite{GM2,Ka2}).\quad{\it
Let $\mu$ be a measurable function. For $1\leq p<\infty$ and
$0\leq \lambda<3$, the Morrey space
$\mathcal{M}^{p,\lambda}(\mathbb{R}^3)$ is defined as
\begin{equation*}
\|\mu\|_{\mathcal{M}^{p,\lambda}(\mathbb{R}^3)}=\left\{\aligned
&\sup_{x\in\mathbb{R}^3,\ r>0}r^{-\lambda}\|\mu\|_{TV(B(x,r))},\ \ p=1,\\
&\sup_{x\in\mathbb{R}^3,\ r>0}r^{-{\lambda}/{p}}\left(\int_{B(x,r)}|\mu|^p(y){\rm d}y\right)^{{1}/{p}},\ \ p>1.
\endaligned\right.
\end{equation*}
Here $B(x,r)$ is the open ball in $\mathbb{R}^3$ with radius
$r$ centered at $x$ and $\|\mu\|_{TV(B(x,r))}$ is the total variation
of $\mu$. Particularly, we denote
$\mathcal{M}^{1,{3}/{q'}}(\mathbb{R}^3)=\mathcal{M}^q(\mathbb{R}^3)$,  for all $1\leq q<\infty$ (see \cite{GM2}).}

\medskip

\quad\hspace{-20pt}{\bf Definition 2.2} (\cite{BCF}).\quad{\it
Let $G(x,t)=(4\pi t)^{-\frac{3}{2}}\exp(-\frac{|x|^2}{4t})$,
$\omega=\nabla\times u$, $j=\nabla\times b$, $u=(u^1,u^2,u^3)$,
$\omega=(\omega^1,\omega^2,\omega^3)$, $j=(j^1,j^2,j^3)$ and
$b=(b^1,b^2,b^3)$. We say that a pair $(\omega,\,j)$ constitutes the
{\it mild solution} to the IVP $(1.5)$ if they satisfy: for all
$0<t<\infty$, $x\in\mathbb{R}^3$ and $i\in\{1,2,3\}$,
$$\begin{array}{ll}
\left\{\aligned
\omega(x,t)&=G(\cdot,t)\ast\omega_0(x)\\
&\quad\displaystyle+\int_{0}^t\int_{\mathbb{R}^3}G(x-y,t-s)\partial_{y_i}( u^i\omega-u\omega^i-b^ij+bj^i)(y,s){\rm d}y{\rm d}s,\\
j(x,t)&=G(\cdot,t)\ast j_0(x)\\
&\quad\displaystyle+\int_{0}^t\int_{\mathbb{R}^3}G(x-y,t-s)
\nabla_y\times((u\cdot\nabla)b-(b\cdot\nabla)u)(y,s){\rm d}y{\rm d}s
\endaligned\right.
\end{array}\eqno(2.1)$$
and
$$\omega(x,t)\xrightarrow{weak^*}\omega_0(x)\quad \text{and} \quad j(x,t)\xrightarrow{weak^*}j_0(x) \quad as \quad t \rightarrow 0.\eqno(2.2)$$
Here,  formula $(2.1)$ is understood in the sense of $L^1(\mathbb{R}^3)$
and formula $(2.2)$ means that for any given
$\phi(x)\in C^\infty_0(\mathbb{R}^3)$, the functions
$$\begin{array}{ll}
g(t)=\left\{\aligned
&\int_{\mathbb{R}^3}\omega(x,t)\phi(x){\rm d}x,\quad t>0,\\
&\int_{\mathbb{R}^3} \phi(x){\rm d}\omega_0,\quad t=0,\endaligned\right.
\end{array}\eqno(2.3)$$
and
$$\begin{array}{ll}
h(t)=\left\{\aligned
&\int_{\mathbb{R}^3}j(x,t)\phi(x){\rm d}x,\quad t>0,\\
&\int_{\mathbb{R}^3} \phi(x){\rm d}j_0,\quad t=0,
\endaligned\right.
\end{array}\eqno(2.4)$$
are both  continuous at $t=0$}.

Here and throughout this paper, the letter $C$ or $c$, sometimes
with certain parameters, will stand for positive constants not
necessarily the same one at each occurrence, but are independent
of the essential variables. We set $\mathbb{R}^{+}:=(0,\infty)$.
For $1\leq p\leq\infty$, we denote $p'$ by the conjugate index
of $p$, that is $1/p+1/p'=1$ (here we set $1'=\infty$ and
$\infty'=1$). We also adopt the following simplified notations:
\begin{equation*}
\|f\|_p=\|f\|_{\mathcal{M}^{p}(\mathbb{R}^3)},\ \ \ \|f\|_{p,\lambda}=\|f\|_{\mathcal{M}^{p,\lambda}(\mathbb{R}^3)},\ \ \|f\|_\infty=\|f\|_{L^\infty(\mathbb{R}^3)}.
\end{equation*}
We assume that all the functions in this paper are real-valued.

\bigskip

\quad\hspace{-20pt}{\bf 2.2. Main results}

\medskip

We now state our main results as follows. The Global-in-time well-posedness of $(1.5)$ can be formulated
as follows:

\begin{theorem}\label{thm2.2}
{\bf (Global-in-time well-posedness).}\quad We define
\begin{equation*}
\begin{array}{ll}
&A_1:=\{(p,q)|p\in(1,2),\ \ q\in[1,3-p), \ \ (p-2)(q-4)\leq2,\ \ \ 3<2p+q<6\},\\
&A_2:=\Big\{(p,q)|q\in[1,2),\ \ p\in\Big(\frac{2(3-q)}{4-q},3-q\Big)\Big\}.
\end{array}
\end{equation*}
Assume that
$$\|\omega_0\|_{1,1}\ \ {\rm and}\ \ \| j_0\|_{1,1}\  {\rm are\ small\ enough},$$
then
\begin{itemize}
\item[(i)] if $(p,q)\in A_1$, there exists a  unique mild solution
$(\omega,j)$ in $[0,\infty)\times \mathbb{R}^3$ to the problem $(1.5)$
such that
$$\begin{array}{ll}
&t^{1-\frac{3-q}{2p}}\omega,\ \ t^{1-\frac{3-q}{2p}}j\in L^\infty(\mathbb{R}^{+};\mathcal{M}^{p,q}(\mathbb{R}^3)),\vspace{1ex}\\
&\omega,\ \ j\in L^{\frac{2p}{2p-3+q}}(\mathbb{R}^{+};\mathcal{M}^{p,q}(\mathbb{R}^3)),\vspace{1ex}\\
&t^{\frac{3}{2}-\frac{3-q}{2p}}\nabla\omega,\ \ t^{\frac{3}{2}-\frac{3-q}{2p}}\nabla j\in L^\infty(\mathbb{R}^{+};\mathcal{M}^{p,q}(\mathbb{R}^3)),\vspace{1ex}\\
&\nabla\omega,\ \ \nabla j\in L^{\frac{2p}{3p-3+q}}(\mathbb{R}^{+};\mathcal{M}^{p,q}(\mathbb{R}^3)),
\end{array}$$
and  the solution  $(\omega,j)$ solves $(1.5)$ in the classical
sense for $t>0$;

\item[(ii)] if $(p,q)\in A_1\cap A_2$,  there exists a unique mild
solution $(\omega,j)$ in $[0,\infty)\times \mathbb{R}^3 $ to the problem
$(1.5)$ such that
$$\begin{array}{ll}
&t^{\frac{q-1}{2}}\omega,\ \ t^{\frac{q-1}{2}}j\in L^\infty(\mathbb{R}^{+};\mathcal{M}^{1,q}(\mathbb{R}^3)),\vspace{1ex}\\
&\omega,\ \ j\in L^{\frac{2}{q-1}}(\mathbb{R}^{+};\mathcal{M}^{1,q}(\mathbb{R}^3)),\vspace{1ex}\\
&t^{\frac{q}{2}}\nabla\omega,\ \ t^{\frac{q}{2}}\nabla j\in L^\infty(\mathbb{R}^{+};\mathcal{M}^{1,q}(\mathbb{R}^3)),\vspace{1ex}\\
&\nabla\omega,\ \ \nabla j\in L^{\frac{2}{q}}(\mathbb{R}^{+};\mathcal{M}^{1,q}(\mathbb{R}^3)),
\end{array}$$
and the solution $(\omega,j)$ solves $(1.5)$ in the classical sense
for $t>0$.
\end{itemize}
\end{theorem}
\begin{remark}\label{rem1}
By (ii) of Theorem \ref{thm2.2}, taking $q=1$, it is easy to see that
there exists a global unique mild solution $(\omega,j)$ in
$[0,\infty)\times \mathbb{R}^3 $ to the problem $(1.5)$ such that
$$\begin{array}{ll}
&\omega,\ \ j\in L^\infty(\mathbb{R}^{+};\mathcal{M}^{1,1}(\mathbb{R}^3)),\vspace{1ex}\\
&t^{\frac{1}{2}}\nabla\omega,\ \ t^{\frac{1}{2}}\nabla j\in L^\infty(\mathbb{R}^{+};\mathcal{M}^{1,1}(\mathbb{R}^3)),\vspace{1ex}\\
&\nabla\omega,\ \ \nabla j\in L^{2}(\mathbb{R}^{+};\mathcal{M}^{1,1}(\mathbb{R}^3)),
\end{array}$$
when $\|\omega_0\|_{1,1}$ and $\| j_0\|_{1,1}$ are small enough.
For this case, the result obtained above can be regarded as an
extension of the theory for {\bf NS} equations in \cite{GM2} to
{\bf MHD}. However, our proof indeed needs some new ideas due
to the strong coupling between the fluid and the magnetic field.
\end{remark}

\bigskip

\quad\hspace{-20pt}{\bf 2.3. Main strategies}

\medskip

The main purpose of the current work is to prove the global-in-time
well-posedness of mild solutions to the Cauchy problem $(1.5)$ , which
generated a sequence from the iterative scheme of the vorticity
equations and converged to a root by the Fixed-point theorem. Actually,
in order to eliminate the total kinetic pressure term $\nabla P_T$ in
$(1.1)$, one usual method is to apply the Leray projector $\mathbb{P}$
in \cite{Ka2} ($\mathbb{P}$ is a matrix $3\times3$ with elements
$(\mathbb{P})_{k,j}=\delta_{kj}+\mathcal{R}_k\mathcal{R}_j$, where
$\mathcal{R}_k=\partial_k(-\triangle)^{{1}/{2}}\,(k=1,2,3)$ are the
Riesz transforms) or the curl operator in \cite{GM2}. Considering the
unboundedness of the Leray projector $\mathbb{P}$ in
$\mathcal{M}^{p}(\mathbb{R}^3)$, we have to use the curl operator in
this paper to handle $\nabla P_T$ along the spirit of  \cite{GM2}.

However, the approaches used in \cite{GM2} for establishing the
existence of the NS system fail to apply to the corresponding problem
of magnetohydrodynamic system directly due to some new mathematical
difficulties:
\begin{itemize}
\item the strong coupling between the magnetic filed $b$ and the fluid
velocity $u$;\\

\item
one can not control $\|\nabla u\|_{p}$ and $\|\nabla b\|_{p}$ by
$\|\nabla\times u\|_{p}$ and $\|\nabla\times b\|_{p}$ in $\mathcal{M}^{p}(\mathbb{R}^3)$,
\end{itemize}
which are  mainly reflected in estimating the cross term $(u\cdot\nabla)b$
and $(b\cdot\nabla)u$.

Hence, compared with the theory for pure NS system, in order to
overcome the difficulties mentioned above, some new observations
are indeed required. Actually, we found that the following
estimate plays a key role in establishing the global-in-time
estimates for the desired  mild solutions:
\begin{equation*}
\|(b\cdot\nabla)u\|_{1,\lambda}\leq\|\nabla u\|_{\theta,\tau}\|b\|_{r,s}\leq\|\omega\|_{\theta,\tau}\|b\|_{r,s}, \quad \mbox{ for \ }1=\frac{1}{\theta}+\frac{1}{r}\mbox{\ and \ } \lambda=\frac{\tau}{\theta}+\frac{s}{r},
\end{equation*}
which means that we need to estimate $\|\omega\|_{\theta,\tau}$ and
$\|b\|_{r,s}$. For this purpose,

\begin{itemize}
\item first, we estimate $\|\omega\|_{\theta,\tau}$ when the initial
norms $\|\omega_0\|_{1,1}$ and $\| j_0\|_{1,1}$ are both sufficiently
small. Lemma 3 in the Appendix B is introduced for this step.
Particularly, we established the bounds for the map
$$T:\mathcal{M}^{q_1,\lambda_1}(\mathbb{R}^3)\rightarrow
\mathcal{M}^{q_2,\lambda_2}(\mathbb{R}^3)$$
with $0\leq\lambda_1\leq\lambda_2<3$, which is new for the case
$0\leq\lambda_1<\lambda_2<3$, if $T$ is convolution operator with
heat kernel (see Proposition 4 in Appendix A);\\

\item second, under the same initial data, we can control $\|b\|_{r,s}$
by $\|\omega\|_{\theta,\tau}$ and $\|\omega\|_{1,\lambda}$.

\end{itemize}

At last, we used a inequality that is similar to Lemma 3 in the
Appendix B to get the well-posedness of the desired  global solution.

\bigskip

\quad\hspace{-20pt}{\bf 2.4. Outline }

\medskip

The rest of the paper is organized as follows. In Section \ref{S3}, we give the proof of Theorem  \ref{thm3.1}, which is divided into the following two steps: 

\begin{itemize}
\item Step (i): in Subsection 3.1, we establish the global-in-time
well-posedness of mild solutions in $\mathcal{M}^{p,q}(\mathbb{R}^3)$
for some $p>1$ and $q\in[0,3)$;\\

\item Step (ii): in Subsection 3.2, we establish the global-in-time
well-posedness of mild solutions in $\mathcal{M}^{1,q}(\mathbb{R}^3)$
for some $q\in[0,4)$.
\end{itemize}
The proof of Theorem \ref{thm2.2}, which can be regarded as the special case of Theorem  \ref{thm3.1}, is listed at the end of the Section \ref{S3}. 

In Appendix,  we give some basic properties of Morrey spaces and some technical lemmas. Among those,  Appendix A is devoted to presenting some basic properties of Morrey spaces and
some inequalities for Riesz potential and convolution operators
with heat kernel and Biot-Savart kernel on Morrey spaces, and
Appendix B is established to give some preliminary
lemmas, which play an important role in our proof.

\section{Proof of Theorem \ref{thm2.2}}\label{S3}
This section will give the proof of Theorem \ref{thm2.2}. Actually,
Theorem \ref{thm2.2} can be deduced by the following result, which
is of interest in its own right.

\begin{theorem}\label{thm3.1}
We set
\begin{equation*}
\begin{array}{ll}
&E_1:=\{(p,q,p_0,q_0)|p_0\in[1,\infty),\ \ p_0\leq p,\ q_0\in[0,3),\ \ 2p_0+q_0=3,\vspace{1ex}\\
&\qquad\quad1<p,\ \ q_0\leq q<3,\ \ p+q<3,\ \ 3<2p+q<6,\ \ (q-4)(p-2)\leq2\};\vspace{2ex}\\
&E_2:=\Big\{(p,q,p_0,q_0,\tilde{q_0},q_1)|(p,q,p_0,q_0)\in E_1,\ {\rm there\ exist}\ q_2,\,q_3\in[0,3)\vspace{1ex}\\
&\qquad\quad{\rm and}\ \tilde{p}\in(1,\min\{p,p'\})\ {\rm such\ that}\ 0\leq q_1-\tilde{q_0}<1,\ 0\leq q_1-q_2<1,\vspace{1ex}\\
&\qquad\quad q_2=\frac{q_3}{p'}+\frac{q}{p},\ \frac{1}{\tilde{p}}=\frac{1}{p'}+\frac{1}{3-q_3},\ \ \frac{q_3}{\tilde{p}}=q_1\big(\frac{p'}{\tilde{p}}-\frac{p'}{p}\big)+\frac{q}{p}\big(p'-\frac{p'}{\tilde{p}}\big),\vspace{1ex}\\
&\qquad\quad\frac{q_2-\tilde{q_0}+1}{2}=\frac{q_1-\tilde{q_0}}{2}\big(\frac{p'}{\tilde{p}}-\frac{p'}{p}\big)+\frac{2p-3+q}{2p}(1+p'-\frac{p'}{\tilde{p}}\big)\Big\}.
\end{array}
\end{equation*}

\begin{itemize}
\item [(i)] Let $(p,q,p_0,q_0)\in E_1$. Assume that
$$\|\omega_0\|_{p_0,q_0}\ \ {\rm and}\ \ \| j_0\|_{p_0,q_0}\ \ {\rm are\ small\ enough}.$$
Then there exists a global unique mild solution $(\omega,j)$ on
$\mathbb{R}^{+}$ of $(1.5)$ such that
$$\begin{array}{ll}
&t^{1-\frac{3-q}{2p}}\omega,\ \ t^{1-\frac{3-q}{2p}}j\in L^\infty(\mathbb{R}^{+};\mathcal{M}^{p,q}(\mathbb{R}^3)),\vspace{1ex}\\
&\omega,\ \ j\in L^{\frac{2p}{2p-3+q}}(\mathbb{R}^{+};\mathcal{M}^{p,q}(\mathbb{R}^3)),\vspace{1ex}\\
&t^{\frac{3}{2}-\frac{3-q}{2p}}\nabla\omega,\ \ t^{\frac{3}{2}-\frac{3-q}{2p}}\nabla j\in L^\infty(\mathbb{R}^{+};\mathcal{M}^{p,q}(\mathbb{R}^3)),\vspace{1ex}\\
&\nabla\omega,\ \ \nabla j\in L^{\frac{2p}{3p-3+q}}(\mathbb{R}^{+};\mathcal{M}^{p,q}(\mathbb{R}^3)).
\end{array}\eqno(3.1)$$
\item[(ii)] Let $(p,q,p_0,q_0,\tilde{q_0},q_1)\in E_2$. Assume
that
$$\|\omega_0\|_{p_0,q_0}\ \ {\rm and}\ \ \| j_0\|_{p_0,q_0}\ \ {\rm are\ small\ enough},$$
and
$$\omega_0,\,j_0\in\mathcal{M}^{1,\tilde{q_0}}(\mathbb{R}^3).$$
Then there exists a global unique mild solution $(\omega,j)$ on
$\mathbb{R}^{+}$ of $(1.5)$ such that
$$\begin{array}{ll}
&t^{\frac{q_1-\tilde{q_0}}{2}}\omega,\ \ t^{\frac{q_1-\tilde{q_0}}{2}}j\in L^\infty(\mathbb{R}^{+};\mathcal{M}^{1,q_1}(\mathbb{R}^3)),\vspace{1ex}\\
&\omega,\ \ j\in L^{\frac{2}{q_1-\tilde{q_0}}}(\mathbb{R}^{+};\mathcal{M}^{1,q_1}(\mathbb{R}^3)),\vspace{1ex}\\
&t^{\frac{1+q_1-\tilde{q_0}}{2}}\nabla\omega,\ \ t^{\frac{1+q_1-\tilde{q_0}}{2}}\nabla j\in L^\infty(\mathbb{R}^{+};\mathcal{M}^{1,q_1}(\mathbb{R}^3)),\vspace{1ex}\\
&\nabla\omega,\ \ \nabla j\in L^{\frac{2}{1+q_1-\tilde{q_0}}}(\mathbb{R}^{+};\mathcal{M}^{1,q_1}(\mathbb{R}^3)).
\end{array}\eqno(3.2)$$
\end{itemize}
\end{theorem}

\bigskip

In order to prove Theorem \ref{thm3.1}, we need to use the standard
successive approximation scheme:
$$\begin{array}{ll}
\left\{\aligned
&\omega^{(0)}(x,t)=G(\cdot,t)*\omega_0(x),\ \ \ j^{(0)}(x,t)=G(\cdot,t)*j_0(x),\\
&\omega^{(k+1)}(x,t)=\omega^{(0)}(x,t)+\displaystyle\int_0^t\int_{\mathbb{R}^3}G(x-y,t-s)\\
&\qquad\qquad\times\partial_{y_i}( {u}^{i,(k)}\omega^{(k)}-u^{(k)}\omega^{i,(k)}-b^{i,(k)}j^{(k)}+b^{(k)}j^{i,(k)})(y,s){\rm d}y{\rm d}s,\\
&j^{(k+1)}(x,t)=j^{(0)}(x,t)+\displaystyle\int_0^t\int_{\mathbb{R}^3}G(x-y,t-s)\\
&\qquad\qquad\times\nabla_y\times((u^{(k)}\cdot\nabla)b^{(k)}-(b^{(k)}\cdot\nabla)u^{(k)})(y,s){\rm d}y{\rm d}s,\\
&u^{(k)}(x,t)=K*\omega^{(k)}(x,t),\ \ b^{(k)}(x,t)=K*j^{(k)}(x,t).
\endaligned\right.
\end{array}\eqno(3.3)$$

We now prove Theorem \ref{thm3.1} by the following two parts:

\medskip

\quad\hspace{-20pt}{\bf 3.1. Proof for part (i) of Theorem \ref{thm3.1}}

\medskip

Let $(p,q,p_0,q_0)\in E_1$. There exist two real numbers $r,\,\theta$
such that
$$1\leq\theta<p,\ \ \frac{1}{\theta}=\frac{1}{r}+\frac{1}{p},\ \ \frac{1}{p}-\frac{1}{r}=\frac{1}{3-q}.$$
For convenience, we set
\begin{equation*}
\begin{array}{ll}
&W_{k,p,q}^{0}=\sup\limits_{t\in\mathbb{R}^{+}}t^{1-\frac{3-q}{2p}}\|\omega^{(k)}(\cdot,t)\|_{p,q},\ \ \ \ \ \ \ \bar{W}_{k,p,q}^{0}=\|\|\omega^{(k)}(\cdot,t)\|_{p,q}\|_{L_t^{\frac{2p}{2p-3+q}}(\mathbb{R}^{+})},\\
&W_{k,p,q}^{1}=\sup\limits_{t\in\mathbb{R}^{+}}t^{\frac{3}{2}-\frac{3-q}{2p}}\|\nabla \omega^{(k)}(\cdot,t)\|_{p,q},\ \ \ \ \bar{W}_{k,p,q}^{1}(u)=\|\|\nabla\omega^{(k)}(\cdot,t)\|_{p,q}\|_{L_t^{\frac{2p}{3p-3+q}}(\mathbb{R}^{+})},\\
&J_{k,p,q}^{0}=\sup\limits_{t\in\mathbb{R}^{+}}t^{1-\frac{3-q}{2p}}\|j^{(k)}(\cdot,t)\|_{p,q},\ \ \ \ \ \ \ \ \ \bar{J}_{k,p,q}^{0}=\|\|j^{(k)}(\cdot,t)\|_{p,q}\|_{L_t^{\frac{2p}{2p-3+q}}(\mathbb{R}^{+})},\\
&J_{k,p,q}^{1}=\sup\limits_{t\in\mathbb{R}^{+}}t^{\frac{3}{2}-\frac{3-q}{2p}}\|\nabla j^{(k)}(\cdot,t)\|_{p,q},\ \ \ \ \ \ \bar{J}_{k,p,q}^{1}(u)=\|\|\nabla j^{(k)}(\cdot,t)\|_{p,q}\|_{L_t^{\frac{2p}{3p-3+q}}(\mathbb{R}^{+})}.
\end{array}
\end{equation*}

The proof for part (i) of Theorem \ref{thm3.1} will be divided into
five steps:

\medskip

{\bf Step 1: Estimates for  the terms $W_{k,p,q}^{0}$, $J_{k,p,q}^{0}$,
$\bar{W}_{k,p,q}^{0}$ and $\bar{J}_{k,p,q}^{0}$.}

\medskip
By (3.3) and Proposition 4, there exists a constant $C_1>0$ such that
$$\begin{array}{ll}
&\|\omega^{(k+1)}(\cdot,t)\|_{p,q}
\leq\displaystyle\|G(\cdot,t)\ast \omega_0\|_{p,q}+\displaystyle C_1\int_0^t(t-s)^{-\frac{1}{2}-\frac{1}{2}(\frac{3-q}{\theta}-\frac{3-q}{p})}\\
&\qquad\qquad\qquad\qquad\times\|({u}^{i,(k)}\omega^{(k)}-u^{(k)}\omega^{i,(k)}-b^{i,(k)}j^{(k)}+b^{(k)}j^{i,(k)})(\cdot,s)\|_{\theta,q}{\rm d}s,
\end{array}\eqno(3.4)$$
and
$$\begin{array}{ll}
&\|j^{(k+1)}(\cdot,t)\|_{p,q}\leq\displaystyle\|G(\cdot,t)\ast j_0\|_{p,q}+C_1\int_0^t(t-s)^{-\frac{1}{2}-\frac{1}{2}(\frac{3-q}{\theta}-\frac{3-q}{p})}\\
&\qquad\qquad\qquad\qquad\times\|((u^{(k)}\cdot\nabla)b^{(k)}-(b^{(k)}\cdot\nabla)u^{(k)})(\cdot,s)\|_{\theta,q}{\rm d}s.
\end{array}\eqno(3.5)$$
Due to  part (i) of  Proposition 2 and Propositions  3-4,  there exists a constant $C_2>0$ such that
$$\begin{array}{ll}
&\quad\|({u}^{i,(k)}\omega^{(k)}-u^{(k)}\omega^{i,(k)}-b^{i,(k)}j^{(k)}+b^{(k)}j^{i,(k)})(\cdot,s)\|_{\theta,q}\vspace{1ex}\\
&\leq\|u^{(k)}(\cdot,s)\|_{r,q}\|\omega^{(k)}(\cdot,s)\|_{p,q}+\|b^{(k)}(\cdot,s)\|_{r,q}\|j^{(k)}(\cdot,s)\|_{p,q}\vspace{1ex}\\
&\leq C_2(\|\omega^{(k)}(\cdot,s)\|_{p,q}^2+\|j^{(k)}(\cdot,s)\|_{p,q}^2),
\end{array}\eqno(3.6)$$
$$\begin{array}{ll}
&\quad\|((u^{(k)}\cdot\nabla)b^{(k)}-(b^{(k)}\cdot\nabla)u^{(k)})(\cdot,s)\|_{\theta,q}\vspace{1ex}\\
&\leq\|\nabla b^{(k)}(\cdot,s)\|_{p,q}\|u^{(k)}(\cdot,s)\|_{r,q}+\|\nabla u^{(k)}(\cdot,s)\|_{p,q}\|b^{(k)}(\cdot,s)\|_{r,q}\vspace{1ex}\\
&\leq\|j^{(k)}(\cdot,s)\|_{p,q}\|u^{(k)}(\cdot,s)\|_{r,q}+\|\omega^{(k)}(\cdot,s)\|_{p,q}\|b^{(k)}(\cdot,s)\|_{r,q}\vspace{1ex}\\
&\leq C_2(\|\omega^{(k)}(\cdot,s)\|_{p,q}^2+\|j^{(k)}(\cdot,s)\|_{p,q}^2).
\end{array}\eqno(3.7)$$
Note that $\frac{1}{2}+\frac{1}{2}(\frac{3-q}{\theta}-\frac{3-q}{p})=
\frac{3-q}{2p}$. Hence, it follows from (3.4)-(3.7) that
$$\begin{array}{ll}
&\|\omega^{(k+1)}(\cdot,t)\|_{p,q}\leq\displaystyle\|G(\cdot,t)\ast \omega_0\|_{p,q}\vspace{1ex}\\
&\qquad\qquad\qquad\qquad\displaystyle+C_1C_2\int_0^t(t-s)^{-\frac{3-q}{2p}}(\|\omega^{(k)}(\cdot,s)\|_{p,q}^2+\|j^{(k)}(\cdot,s)\|_{p,q}^2){\rm d}s,
\end{array}\eqno(3.8)$$
and
$$\begin{array}{ll}
&\|j^{(k+1)}(\cdot,t)\|_{p,q}\leq\displaystyle\|G(\cdot,t)\ast j_0\|_{p,q}\vspace{1ex}\\
&\qquad\qquad\qquad\qquad\displaystyle+C_1C_2\int_0^t(t-s)^{-\frac{3-q}{2p}}(\|\omega^{(k)}(\cdot,s)\|_{p,q}^2+\|j^{(k)}(\cdot,s)\|_{p,q}^2){\rm d}s.
\end{array}\eqno(3.9)$$

Define the functions $\{u_k\}_{k\geq1}$ by the following
\begin{equation*}
\left\{\aligned
&u_1(x,t)=G(\cdot,t)*(|\omega_0|+|j_0|)(x),\vspace{1ex}\\
&u_k(x,t)=(|\omega^{(k)}|+|j^{(k)}|)(x,t)\ \ \ {\rm for\ all}\ k\geq2.
\endaligned\right.
\end{equation*}
Inequality (3.8) together with (3.9) yields that
$$\|u_{k+1}(\cdot,t)\|_{p,q}\leq 2\|G(\cdot,t)*(|\omega_0|+|j_0|)\|_{p,q}+ 4C_1C_2\int_0^t(t-s)^{-\frac{3-q}{2p}}\|u_k(\cdot,s)\|_{p,q}^2{\rm d}s.\eqno(3.10)$$
On the other hand, by Proposition 4, there exists a
constant $A_1>0$ such that
$$\|u_1(\cdot,t)\|_{p,q}\leq A_1 t^{-\frac{1}{2}(\frac{3-q_0}{p_0}-\frac{3-q}{p})}\|u_0\|_{p_0,q_0}\leq A_1 t^{\frac{3-q}{2p}-1}\|u_0\|_{p_0,q_0}\eqno(3.11)$$
for all $u_0\in\mathcal{M}^{p_0,q_0}(\mathbb{R}^3)$. By (3.10), (3.11)
and invoking Lemma 3, we can get
$$W_{k,p,q}^0+J_{k,p,q}^0+\bar{W}_{k,p,q}^0+\bar{J}_{k,p,q}^0
\leq 16A_1\||\omega_0|+|j_0|\|_{p_0,q_0}\leq 16A_1(\|\omega_0\|_{p_0,q_0}+\|j_0\|_{p_0,q_0})\eqno(3.12)$$
for all $k\geq1$, whenever
$$\|\omega_0\|_{p_0,q_0}+\|j_0\|_{p_0,q_0}\leq\frac{1}{32A_1C_1C_2}
\min\Big\{\frac{1}{\mathcal{C}(1-\frac{3-q}{2p},\frac{3-q}{p}-1)},1\Big\}=:G_1.\eqno(3.13)$$

{\bf Step 2: Estimates for  the terms $W_{k,p,q}^1$ and $J_{k,p,q}^1$.}

\medskip

By (3.3) and Proposition 4 again, there exists a constant $C_3>0$
such that
$$\begin{array}{ll}
&\quad\|\nabla\omega^{(k+1)}(\cdot,t)\|_{p,q}\\
&\leq\displaystyle\|\nabla G(\cdot,t)\ast\omega_0\|_{p,q}+C_3\int_{{t}/{2}}^t(t-s)^{-\frac{1}{2}-\frac{1}{2}(\frac{3-q}{\theta}-\frac{3-q}{p})}\\
&\quad\times\|((u^{(k)}\cdot\nabla)\omega^{(k)}-(\omega^{(k)}\cdot\nabla)u^{(k)}-(b^{(k)}\cdot\nabla)j^{(k)}+(j^{(k)}\cdot\nabla)b^{(k)})(\cdot,s)\|_{\theta,q}{\rm d}s\\
&\quad+\displaystyle C_3\int_0^{{t}/{2}}(t-s)^{-1-\frac{1}{2}(\frac{3-q}{\theta}-\frac{3-q}{p})}\\
&\quad\times\|(u^{i,(k)}\omega^{(k)}-u^{(k)}\omega^{i,(k)}-b^{i,(k)}j^{(k)}+b^{(k)}j^{i,(k)})(\cdot,s)\|_{\theta,q}{\rm d}s,
\end{array}\eqno(3.14)$$
$$\begin{array}{ll}
&\quad\|\nabla j^{(k+1)}(\cdot,t)\|_{p,q}\\
&\leq\displaystyle\|\nabla G(\cdot,t)\ast j_0\|_{p,q}+C_3\int_{{t}/{2}}^t(t-s)^{-\frac{1}{2}-\frac{1}{2}(\frac{3-q}{\theta}-\frac{3-q}{p})}\\
&\quad\times\|((u^{(k)}\cdot\nabla)j^{(k)}+\nabla u^{i,(k)}\times b^{(k)}_{x_i}-(b^{(k)}\cdot\nabla)\omega^{(k)}-\nabla b^{i,(k)}\times u^{(k)}_{x_i})(\cdot,s)\|_{\theta,q}{\rm d}s\\
&\quad+\displaystyle C_3\int_0^{{t}/{2}}(t-s)^{-1-\frac{1}{2}(\frac{3-q}{\theta}-\frac{3-q}{p})}
\|((u^{(k)}\cdot\nabla)b^{(k)}-(b^{(k)}\cdot\nabla)u^{(k)})(\cdot,s)\|_{\theta,q}{\rm d}s.
\end{array}\eqno(3.15)$$
By Propositions 2 (i) and 3, one finds that
$$\begin{array}{ll}
&\quad\|((u^{(k)}\cdot\nabla)\omega^{(k)}-(\omega^{(k)}\cdot\nabla)u^{(k)}-(b^{(k)}\cdot\nabla)j^{(k)}+(j^{(k)}\cdot\nabla)b^{(k)})(\cdot,s)\|_{\theta,q}\vspace{1ex}\\
&\leq\|u^{(k)}(\cdot,s)\|_{r,q}\|\nabla\omega^{(k)}(\cdot,s)\|_{p,q}+\|\omega^{(k)}(\cdot,s)\|_{p,q}\|\nabla u^{(k)}(\cdot,s)\|_{r,q}\vspace{1ex}\\
&\quad+\|b^{(k)}(\cdot,s)\|_{r,q}\|\nabla j^{(k)}(\cdot,s)\|_{p,q}+\|j^{(k)}(\cdot,s)\|_{p,q}\|\nabla b^{(k)}(\cdot,s)\|_{r,q}\vspace{1ex}\\
&\leq2C_2(\|\omega^{(k)}(\cdot,s)\|_{p,q}\|\nabla\omega^{(k)}(\cdot,s)\|_{p,q}+\|j^{(k)}(\cdot,s)\|_{p,q}\|\nabla j^{(k)}(\cdot,s)\|_{p,q}),
\end{array}\eqno(3.16)$$
$$\begin{array}{ll}
&\quad\|((u^{(k)}\cdot\nabla)j^{(k)}+\nabla u^{i,(k)}\times b^{(k)}_{x_i}-(b^{(k)}\cdot\nabla)\omega^{(k)}-\nabla b^{i,(k)}\times u^{(k)}_{x_i})(\cdot,s)\|_{\theta,q}\vspace{1ex}\\
&\leq\|u^{(k)}(\cdot,s)\|_{r,q}\|\nabla j^{(k)}(\cdot,s)\|_{p,q}+2\|\nabla u^{(k)}(\cdot,s)\|_{p,q}\|\nabla b^{(k)}(\cdot,s)\|_{r,q}\vspace{1ex}\\
&\quad+\|b^{(k)}\|_{r,q}\|\nabla\omega^{(k)}(\cdot,s)\|_{p,q}\vspace{1ex}\\
&\leq C_2(3\|\omega^{(k)}(\cdot,s)\|_{p,q}\|\nabla j^{(k)}(\cdot,s)\|_{p,q}+\|j^{(k)}(\cdot,s)\|_{p,q}\|\nabla\omega^{(k)}(\cdot,s)\|_{p,q}).
\end{array}\eqno(3.17)$$
Applying Proposition 4, there exists a constant $A_2>0$ such that
$$\|\nabla G(\cdot,t)*u_0\|_{p,q}\leq A_2t^{-\frac{1}{2}(\frac{3-q_0}{p_0}-\frac{3-q}{p}+1)}\|u_0\|_{p_0,q_0}=A_2 t^{\frac{3-q}{2p}-\frac{3}{2}}\|u_0\|_{p_0,q_0}.\eqno(3.18)$$
It now follows from (3.14)-(3.18), (3.6) and (3.7) that
$$\begin{array}{ll}
&\quad\|\nabla\omega^{(k+1)}(\cdot,t)\|_{p,q}\\
&\leq\displaystyle A_2 t^{\frac{3-q}{2p}-\frac{3}{2}}\|\omega_0\|_{p_0,q_0}
+\displaystyle2C_2C_3\int_{{t}/{2}}^t(t-s)^{-\frac{3-q}{2p}}(\|\omega^{(k)}(\cdot,s)\|_{p,q}\|\nabla\omega^{(k)}(\cdot,s)\|_{p,q}\vspace{1ex}\\
&\quad+\|j^{(k)}(\cdot,s)\|_{p,q}\|\nabla j^{(k)}(\cdot,s)\|_{p,q})ds\vspace{1ex}\\
&\quad+\displaystyle C_2C_3\int_0^{{t}/{2}}(t-s)^{-\frac{1}{2}-\frac{3-q}{2p}}(\|\omega^{(k)}(\cdot,s)\|_{p,q}^2+\|j^{(k)}(\cdot,s)\|_{p,q}^2){\rm d}s,
\end{array}\eqno(3.19)$$
$$\begin{array}{ll}
&\quad\|\nabla j^{(k+1)}(\cdot,t)\|_{p,q}\\
&\leq\displaystyle A_2 t^{\frac{3-q}{2p}-\frac{3}{2}}\|j_0\|_{p_0,q_0}+C_2C_3\displaystyle\int_{{t}/{2}}^t(t-s)^{-\frac{3-q}{2p}}(3\|\omega^{(k)}(\cdot,s)\|_{p,q}\|\nabla j^{(k)}(\cdot,s)\|_{p,q}\vspace{1ex}\\
&\quad+\|j^{(k)}(\cdot,s)\|_{p,q}\|\nabla\omega^{(k)}(\cdot,s)\|_{p,q}){\rm d}s\vspace{1ex}\\
&\quad+C_2C_3\displaystyle\int_0^{{t}/{2}}(t-s)^{-\frac{1}{2}-\frac{3-q}{2p}}(\|\omega^{(k)}(\cdot,s)\|_{p,q}^2+\|j^{(k)}(\cdot,s)\|_{p,q}^2){\rm d}s.
\end{array}\eqno(3.20)$$
Some direct computations may yield that
$$\begin{array}{ll}
&\quad\displaystyle\int_{{t}/{2}}^t(t-s)^{-\frac{3-q}{2p}}(\|\omega^{(k)}(\cdot,s)\|_{p,q}\|\nabla \omega^{(k)}(\cdot,s)\|_{p,q}+\|j^{(k)}(\cdot,s)\|_{p,q}\|\nabla j^{(k)}(\cdot,s)\|_{p,q}){\rm d}s\\
&\leq\displaystyle(W_{k,p,q}^0W_{k,p,q}^1+J_{k,p,q}^0J_{k,p,q}^1\int_{{t}/{2}}^t(t-s)^{-\frac{3-q}{2p}}s^{\frac{3-q}{p}-\frac{5}{2}}{\rm d}s\vspace{1ex}\\
&\leq\frac{2p}{2p+q-3}2^{\frac{3}{2}-\frac{3-q}{2p}}t^{\frac{3-q}{2p}-\frac{3}{2}}
(W_{k,p,q}^0W_{k,p,q}^1+J_{k,p,q}^0J_{k,p,q}^1),
\end{array}\eqno(3.21)$$
$$\begin{array}{ll}
&\quad\displaystyle\int_{{t}/{2}}^t(t-s)^{-\frac{3-q}{2p}}(3\|\omega^{(k)}(\cdot,s)\|_{p,q}\|\nabla j^{(k)}(\cdot,s)\|_{p,q}+\|j^{(k)}(\cdot,s)\|_{p,q}\|\nabla \omega^{(k)}(\cdot,s)\|_{p,q}){\rm d}s\vspace{1ex}\\
&\leq\frac{2p}{2p+q-3}2^{\frac{3}{2}-\frac{3-q}{2p}}t^{\frac{3-q}{2p}-\frac{3}{2}}
(3W_{k,p,q}^0J_{k,p,q}^1+J_{k,p,q}^0W_{k,p,q}^1),
\end{array}\eqno(3.22)$$
$$\begin{array}{ll}
&\quad\displaystyle\int_0^{{t}/{2}}(t-s)^{-\frac{1}{2}-\frac{3-q}{2p}}(\|\omega^{(k)}(\cdot,s)\|_{p,q}^2+\|j^{(k)}(\cdot,s)\|_{p,q}^2){\rm d}s\\
&\leq\displaystyle (W_{k,p,q}^0W_{k,p,q}^0+J_{k,p,q}^0J_{k,p,q}^0)
\int_0^{t/2}(t-s)^{-\frac{1}{2}-\frac{3-q}{2p}}s^{\frac{3-q}{p}-2}{\rm d}s\vspace{1ex}\\
&\leq\frac{p}{3-q-p}2^{\frac{3}{2}-\frac{3-q}{2p}}t^{\frac{3-q}{2p}-\frac{3}{2}}
(W_{k,p,q}^0W_{k,p,q}^0+J_{k,p,q}^0J_{k,p,q}^0).
\end{array}\eqno(3.23)$$
Inequality (3.19) together with (3.20)-(3.23) implies that there 
exists a constant $C_4>0$ such that
$$W_{k+1,p,q}^1\leq A_2\|\omega_0\|_{p_0,q_0}+C_4(W_{k,p,q}^0W_{k,p,q}^1+J_{k,p,q}^0J_{k,p,q}^1)+C_4(W_{k,p,q}^0W_{k,p,q}^0+J_{k,p,q}^0J_{k,p,q}^0),\eqno(3.24)$$
$$J_{k+1,p,q}^1\leq A_2\|j_0\|_{p_0,q_0}+C_4(3W_{k,p,q}^0J_{k,p,q}^1+J_{k,p,q}^0W_{k,p,q}^1)+C_4(W_{k,p,q}^0W_{k,p,q}^0+J_{k,p,q}^0J_{k,p,q}^0).\eqno(3.25)$$
Therefore, by (3.12), (3.24) and (3.25), we have that, there exists 
$C_5>0$ such that
$$W_{k+1,p,q}^1+J_{k,p,q}^1\leq C_5(\|\omega_0\|_{p_0,q_0}+\|j_0\|_{p_0,q_0})(1+W_{k,p,q}^1+J_{k,p,q}^1)\ \ {\rm for\ all}\ k\geq1\eqno(3.26)$$
whenever
$$\|\omega_0\|_{p_0,q_0}+\|j_0\|_{p_0,q_0}\leq\min\{G_1,1\}.\eqno(3.27)$$
By (3.18), it holds that
$$W_{1,p,q}^1+J_{1,p,q}^1\leq A_2(\|\omega_0\|_{p_0,q_0}+\|j_0\|_{p_0,q_0}).\eqno(3.28)$$
Combining (3.26) with (3.28) implies that
$$\begin{array}{ll}
&\quad W_{k,p,q}^1+J_{k,p,q}^1\vspace{1ex}\\
&\leq\displaystyle\Big(1+\frac{A_2}{C_5}\Big)
\frac{C_5(\|\omega_0\|_{p_0,q_0}+\|j_0\|_{p_0,q_0})}{1-C_5(\|\omega_0\|_{p_0,q_0}+\|j_0\|_{p_0,q_0})}
(1-(C_5(\|\omega_0\|_{p_0,q_0}+\|j_0\|_{p_0,q_0}))^k)\vspace{1ex}\\
&\leq2(A_2+C_5)(\|\omega_0\|_{p_0,q_0}+\|j_0\|_{p_0,q_0})\ \ \ {\rm for\ all}\ k\geq1
\end{array}\eqno(3.29)$$
whenever
$$\|\omega_0\|_{p_0,q_0}+\|j_0\|_{p_0,q_0}\leq\min\{G_1,1,(2C_5)^{-1}\}.\eqno(3.30)$$

{\bf Step 3: Estimates for  the terms $\bar{W}_{k,p,q}^1$ and $\bar{J}_{k,p,q}^1$.}

\medskip

By (3.3) and Proposition 4, there exists a constant $C_6>0$ such that
$$\begin{array}{ll}
&\quad\|\nabla\omega^{(k+1)}(\cdot,t)\|_{p,q}\\
&\leq\displaystyle\|\nabla G(\cdot,t)\ast\omega_0\|_{p,q}+C_6\int_{0}^t(t-s)^{-\frac{1}{2}-\frac{1}{2}(\frac{3-q}{\theta}-\frac{3-q}{p})}\\
&\quad\times\|((u^{(k)}\cdot\nabla)\omega^{(k)}-(\omega^{(k)}\cdot\nabla)u^{(k)}
-(b^{(k)}\cdot\nabla)j^{(k)}+(j^{(k)}\cdot\nabla)b^{(k)})(\cdot,s)\|_{\theta,q}{\rm d}s,
\end{array}\eqno(3.31)$$
$$\begin{array}{ll}
&\quad\|\nabla j^{(k+1)}(\cdot,t)\|_{p,q}\\
&\leq\displaystyle\|\nabla G(\cdot,t)\ast j_0\|_{p,q}+C_6\int_{0}^t(t-s)^{-\frac{1}{2}-\frac{1}{2}(\frac{3-q}{\theta}-\frac{3-q}{p})}\\
&\quad\times\|((u^{(k)}\cdot\nabla)j^{(k)}+\nabla u^{i,(k)}\times b^{(k)}_{x_i}-(b^{(k)}\cdot\nabla)\omega^{(k)}-\nabla b^{i,(k)}\times u^{(k)}_{x_i})(\cdot,s)\|_{\theta,q}{\rm d}s.
\end{array}\eqno(3.32)$$
It follows from (3.16), (3.17), (3.31) and (3.32) that
$$\begin{array}{ll}
&\quad\|\nabla\omega^{(k+1)}(\cdot,t)\|_{p,q}\\
&\leq\displaystyle\|\nabla G(\cdot,t)*\omega_0\|_{p,q}+2C_2C_6\int_{0}^t(t-s)^{-\frac{3-q}{2p}}\\
&\quad\times(\|\omega^{(k)}(\cdot,s)\|_{p,q}\|\nabla\omega^{(k)}(\cdot,s)\|_{p,q}+\|j^{(k)}(\cdot,s)\|_{p,q}\|\nabla j^{(k)}(\cdot,s)\|_{p,q}){\rm d}s,
\end{array}\eqno(3.33)$$
$$\begin{array}{ll}
&\quad\|\nabla j^{(k+1)}(\cdot,t)\|_{p,q}\\
&\leq\displaystyle\|\nabla G(\cdot,t)*j_0\|_{p,q}+C_2C_6\int_{0}^t(t-s)^{-\frac{3-q}{2p}}\\
&\quad\times(3\|\omega^{(k)}(\cdot,s)\|_{p,q}\|\nabla j^{(k)}(\cdot,s)\|_{p,q}+\|j^{(k)}(\cdot,s)\|_{p,q}\|\nabla\omega^{(k)}(\cdot,s)\|_{p,q}){\rm d}s.
\end{array}\eqno(3.34)$$
We set
$$a=\Big(\frac{3}{2}-\frac{3-q}{2p}\Big)^{-1},\ \ b=\Big(\frac{5}{2}-\frac{3-q}{p}\Big)^{-1}.$$
Note that $a>1,\,b>1$ and
$$\frac{1}{a}+1=\frac{1}{b}+\frac{3-q}{2p},\ \ \frac{1}{b}=\Big(\frac{3}{2}-\frac{3-q}{2p}\Big)+\Big(1-\frac{3-q}{2p}\Big).$$
By (3.33), (3.34), Proposition 4, Lemma 2, H\"{o}lder's inequality
and the following Young's inequality
$$\|f*g\|_{L^a(\mathbb{R}^{+})}\leq\|f\|_{L^{\frac{2p}{3-q},\infty}(\mathbb{R}^{+})}\|g\|_{L^b(\mathbb{R}^{+})},$$
it holds that
$$\begin{array}{ll}
&\bar{W}_{k+1,p,q}^1
\leq A_3\|\omega_0\|_{p_0,q_0}+2C_2C_6\|t^{-\frac{3-q}{2p}}\|_{L^{\frac{2p}{3-q},\infty}(\mathbb{R}^{+})}\vspace{1ex}\\
&\qquad\qquad\quad\times\|(\|\omega^{(k)}(\cdot,t)\|_{p,q}\|\nabla\omega^{(k)}(\cdot,t)\|_{p,q}\vspace{1ex}\\
&\qquad\qquad+\|j^{(k)}(\cdot,t)\|_{p,q}\|\nabla j^{(k)}(\cdot,t)\|_{p,q})\|_{L^b(\mathbb{R}^{+})}\vspace{1ex}\\
&\qquad\qquad\leq A_3\|\omega_0\|_{p_0,q_0}+2C_2C_6(\bar{W}_{k,p,q}^0\bar{W}_{k,p,q}^1+\bar{J}_{k,p,q}^0\bar{J}_{k,p,q}^1);
\end{array}\eqno(3.35)$$
$$\begin{array}{ll}
&\bar{J}_{k+1,p,q}^1\leq\displaystyle A_3\|j_0\|_{p_0,q_0}+3C_2C_6\|t^{-\frac{3-q}{2p}}\|_{L^{\frac{2p}{3-q},\infty}(\mathbb{R}^{+})}\vspace{1ex}\\
&\qquad\qquad\times\|(\|\omega^{(k)}(\cdot,t)\|_{p,q}\|\nabla j^{(k)}(\cdot,t)\|_{p,q}+\|j^{(k)}(\cdot,t)\|_{p,q}\|\nabla\omega^{(k)}(\cdot,t)\|_{p,q})\|_{L^b(\mathbb{R}^{+})}\vspace{1ex}\\
&\qquad\quad\leq A_3\|j_0\|_{p_0,q_0}+3C_2C_6(3\bar{W}_{k,p,q}^0\bar{J}_{k,p,q}^1+\bar{J}_{k,p,q}^0\bar{W}_{k,p,q}^1).
\end{array}\eqno(3.36)$$
It follow from (3.12), (3.35) and (3.36) that there exists $C_7>0$
such that
$$\bar{W}_{k+1,p,q}^1+\bar{J}_{k+1,p,q}^1\leq C_7(\|\omega_0\|_{p_0,q_0}+\|j_0\|_{p_0,q_0})(1+\bar{W}_{k,p,q}^1+\bar{J}_{k,p,q}^1)\ \ \ {\rm for\ all}\ k\geq1\eqno(3.37)$$
whenever the condition (3.27) holds.

Applying Proposition 4 and Lemma 2, there exists a constant $A_4>0$
such that
$$\bar{W}_{1,p,q}^1+\bar{J}_{1,p,q}^1\leq A_4(\|\omega_0\|_{p_0,q_0}+\|j_0\|_{p_0,q_0}).\eqno(3.38)$$
Combining (4.38) with (4.37) implies that
$$\bar{W}_{k,p,q}^1+\bar{J}_{k,p,q}^1\leq2(A_4+C_7)(\|\omega_0\|_{p_0,q_0}+\|j_0\|_{p_0,q_0})\ \ \ {\rm for\ all}\ k\geq1\eqno(3.39)$$
whenever
$$\|\omega_0\|_{p_0,q_0}+\|j_0\|_{p_0,q_0}\leq\min\{G_1,1,(2C_7)^{-1}\}.\eqno(3.40)$$

{\bf Step 4: Conclusions.}

\medskip

Define the following mappings:
\begin{equation*}
\mathcal{T}_{1}\omega^{(k)}=\omega^{(k+1)},\ \ \ \mathcal{T}_2 j^{(k)}=j^{(k+1)}.
\end{equation*}
Let $\omega_0,\,j_0$ satisfy
$$\|\omega_0\|_{p_0,q_0}+\|j_0\|_{p_0,q_0}\leq G_1.\eqno(3.41)$$
We set $R=16A_1(\|\omega_0\|_{p_0,q_0}+\|j_0\|_{p_0,q_0})$ and
\begin{equation*}
B_R=\Big\{f\in L^{\frac{2p}{2p-3+q}}(\mathbb{R}^{+};\mathcal{M}^{p,q}(\mathbb{R}^3)):\|\|f\|_{p,q}\|_{L^{\frac{2p}{2p-3+q}}(\mathbb{R}^{+})}\leq R\Big\}.
\end{equation*}
By {\bf Step 1} we see that $\{\omega^{(k)}\}\subset B_R$ and
$\{j^{(k)}\}\subset B_R$. Moreover,
$L^{\frac{2p}{2p-3+q}}(\mathbb{R}^{+};\mathcal{M}^{p,q}(\mathbb{R}^3))$
is a Banach space. Applying Banach's fixed point Theorem, there exits
a global mild solution $(\omega,j)$ on $\mathbb{R}^{+}$ of (1.2) such
that the estimates (3.1) hold under the condition (3.41) holds. Moreover,
it holds that
\begin{equation*}
\|\|\omega^{(k)}-\omega\|_{p,q}\|_{L_t^{\frac{2p}{2p-3+q}}(\mathbb{R}^{+})}+
\|\|j^{(k)}-j\|_{p,q}\|_{L_t^{\frac{2p}{2p-3+q}}(\mathbb{R}^{+})}\rightarrow0\ \ {\rm as}\ k\rightarrow\infty.
\end{equation*}

{\bf Step 5: Proof of the uniqueness.}

\medskip

Assume that there exist two global mild solution $(\omega,j)$ and
$(\bar{\omega},\bar{j})$ on $\mathbb{R}^{+}$ of (1.2) satisfy the
estimates (3.1). From {\bf Step 4} we also see that
$$\|\|\omega^{(k)}-\omega\|_{p,q}\|_{L_t^{\frac{2p}{2p-3+q}}(\mathbb{R}^{+})}+
\|\|j^{(k)}-j\|_{p,q}\|_{L_t^{\frac{2p}{2p-3+q}}(\mathbb{R}^{+})}\rightarrow0\ \ {\rm as}\ k\rightarrow\infty,\eqno(3.42)$$
$$\|\|\omega^{(k)}-\bar{\omega}\|_{p,q}\|_{L_t^{\frac{2p}{2p-3+q}}(\mathbb{R}^{+})}+
\|\|j^{(k)}-\bar{j}\|_{p,q}\|_{L_t^{\frac{2p}{2p-3+q}}(\mathbb{R}^{+})}\rightarrow0\ \ {\rm as}\ k\rightarrow\infty,\eqno(3.43)$$
$$\bar{W}_{k,p,q}^0+\bar{J}_{k,p,q}^0\leq 16A_1(\|\omega_0\|_{p_0,q_0}+\|j_0\|_{p_0,q_0})\ \ \ {\rm for\ all}\  k\geq1\eqno(3.44)$$
whenever (3.13) holds.

For convenience, given a function $f:\mathbb{R}^3\times(0,\infty)$,
we set
$$I_{p,q}(f):=\|\|f(\cdot,t)\|_{p,q}\|_{L_t^{\frac{2p}{2p-3+q}}(\mathbb{R}^{+})}.$$
Then (3.42)-(3.44) together with the Minkowski's inequality imply that
$$I_{p,q}(\omega)+I_{p,q}(j)+I_{p,q}(\bar{\omega})+I_{p,q}(\bar{j})\leq 20A_1(\|\omega_0\|_{p_0,q_0}+\|j_0\|_{p_0,q_0}),\eqno(3.45)$$
whenever (3.13) holds.

For convenience, we set
\begin{equation*}
\begin{array}{ll}
&u=K*\omega,\ \ \bar{u}=K*\bar{\omega},\ \ b=K*j,\ \ \bar{b}=K*\bar{j},\vspace{1ex}\\
&\omega(0,x)=\bar{\omega}(0,x)=\omega_0(x),\ \ j(0,x)=\bar{j}(0,x)=j_0(x),\vspace{1ex}\\
&[\omega]=\omega-\bar{\omega},\ \ [j]=j-\bar{j},\ \ [u]=u-\bar{u},\ \ \ [b]=b-\bar{b}.
\end{array}
\end{equation*}
Clearly,
\begin{equation*}
[u]=K*[\omega],\ \ \ [b]=K*[j].
\end{equation*}
By (2.1), it holds that
$$\begin{array}{ll}
&[\omega](x,t)=\displaystyle\int_{0}^t\int_{\mathbb{R}^3}G(x-y,t-s)\partial_{y_i}([u]^i\omega-[u]\omega^i-[b]^ij+[b]j^i)(y,s){\rm d}y{\rm d}s\\
&\qquad\qquad+\displaystyle\int_{0}^t\int_{\mathbb{R}^3}G(x-y,t-s)\partial_{y_i}(\bar{u}^i[\omega]
-\bar{u}[\omega]^i-\bar{b}^i[j]+\bar{b}[j]^i)(y,s){\rm d}y{\rm d}s,
\end{array}\eqno(3.46)$$
$$\begin{array}{ll}
&[j](x,t)=\displaystyle\int_{0}^t\int_{\mathbb{R}^3}G(x-y,t-s)\nabla_y\times(([u]\cdot\nabla)[b]-([b]\cdot\nabla)[u])(y,s){\rm d}y{\rm d}s\\
&\qquad\qquad+\displaystyle\int_{0}^t\int_{\mathbb{R}^3}G(x-y,t-s)\nabla_y\times(([u]\cdot\nabla)\bar{b}-(\bar{b}\cdot\nabla)[u])(y,s){\rm d}y{\rm d}s\\
&\qquad\qquad+\displaystyle\int_{0}^t\int_{\mathbb{R}^3}G(x-y,t-s)\nabla_y\times((\bar{u}\cdot\nabla)[b]-([b]\cdot\nabla)\bar{u})(y,s){\rm d}y{\rm d}s.
\end{array}\eqno(3.47)$$
By Proposition 4, (3.46) and (3.47), there exists a constant $C_8>0$
such that
$$\begin{array}{ll}
&\|[\omega]\|_{p,q}\leq\displaystyle C_8\int_{0}^t(t-s)^{-\frac{3-q}{2p}}\|([u]^i\omega-[u]\omega^i-[b]^ij+[b]j^i\vspace{1ex}\\
&\qquad\qquad\quad+\bar{u}^i[\omega]-\bar{u}[\omega]^i-\bar{b}^i[j]+\bar{b}[j]^i)(\cdot,s)\|_{\theta,q}{\rm d}s,
\end{array}\eqno(3.48)$$
$$\begin{array}{ll}
&\|[j]\|_{p,q}\leq\displaystyle C_8\int_0^t(t-s)^{-\frac{3-q}{2p}}\|\nabla_y\times(([u]\cdot\nabla)[b]-([b]\cdot\nabla)[u])(\cdot,s)\vspace{1ex}\\
&\qquad\qquad+\nabla_y\times(([u]\cdot\nabla)\bar{b}-(\bar{b}\cdot\nabla)[u])(\cdot,s)\vspace{1ex}\\
&\qquad\qquad+\nabla_y\times((\bar{u}\cdot\nabla)[b]-([b]\cdot\nabla)\bar{u})(\cdot,s)\|_{\theta,q}{\rm d}s.
\end{array}\eqno(3.49)$$
By Propositions 2 (i) and 3, there exists a constant $C_9>0$ such that
$$\begin{array}{ll}
&\quad\displaystyle\|([u]^i\omega-[u]\omega^i-[b]^ij+[b]j^i
+\bar{u}^i[\omega]-\bar{u}[\omega]^i-\bar{b}^i[j]+\bar{b}[j]^i)(\cdot,s)\|_{\theta,q}\vspace{1ex}\\
&\leq C_9(\|[\omega]\|_{p,q}\|\omega\|_{p,q}+\|[j]\|_{p,q}\|j\|_{p,q}+\|\bar{\omega}\|_{p,q}\|[\omega]\|_{p,q}+\|\bar{j}\|_{p,q}\|[j]\|_{p,q}),
\end{array}\eqno(3.50)$$
$$\begin{array}{ll}
&\quad\displaystyle \|\nabla_y\times(([u]\cdot\nabla)[b]-([b]\cdot\nabla)[u])(\cdot,s)+\nabla_y\times(([u]\cdot\nabla)\bar{b}-(\bar{b}\cdot\nabla)[u])(\cdot,s)\vspace{1ex}\\
&\quad+\nabla_y\times((\bar{u}\cdot\nabla)[b]-([b]\cdot\nabla)\bar{u})(\cdot,s)\|_{\theta,q}\vspace{1ex}\\
&\leq C_9(\|[\omega]\|_{p,q}\|[j]\|_{p,q}+\|[\omega]\|_{p,q}\|\bar{j}\|_{p,q}+\|\bar{\omega}\|_{p,q}\|[j]\|_{p,q}).
\end{array}\eqno(3.51)$$
It follows from (3.49)-(3.51) that
$$\begin{array}{ll}
&\|[\omega]\|_{p,q}\leq\displaystyle C_8C_9\int_{0}^t(t-s)^{-\frac{3-q}{2p}}(\|[\omega]\|_{p,q}\|\omega\|_{p,q}+\|[j]\|_{p,q}\|j\|_{p,q}\vspace{1ex}\\
&\qquad\qquad+\|\bar{\omega}\|_{p,q}\|[\omega]\|_{p,q}+\|\bar{j}\|_{p,q}\|[j]\|_{p,q}){\rm d}s,
\end{array}\eqno(3.52)$$
$$\|[j]\|_{p,q}\leq\displaystyle C_8C_9\int_0^t(t-s)^{-\frac{3-q}{2p}}(\|[\omega]\|_{p,q}\|[j]\|_{p,q}
+\|[\omega]\|_{p,q}\|\bar{j}\|_{p,q}+\|\bar{\omega}\|_{p,q}\|[j]\|_{p,q}){\rm d}s.\eqno(3.53)$$
By the H\"{o}lder's inequality and the following Young's inequality
$$\|f*g\|_{L^a(\mathbb{R}^{+})}\leq\|f\|_{L^{\frac{2p}{3-q},\infty}(\mathbb{R}^{+})}\|g\|_{L^b(\mathbb{R}^{+})}$$
with $a=\frac{2p}{2p-3+q}$ and $b=\frac{p}{2p-3+q}$, we get from
(3.52)-(3.53) that
$$\begin{array}{ll}
&I_{p,q}([\omega])\leq\displaystyle C_8C_9\|t^{-\frac{3-q}{2p}}\|_{L^{\frac{2p}{3-q},\infty}(\mathbb{R}^{+})}\|\|[\omega]\|_{p,q}\|\omega\|_{p,q}+\|[j]\|_{p,q}\|j\|_{p,q}\vspace{1ex}\\
&\qquad\qquad\quad+\|\bar{\omega}\|_{p,q}\|[\omega]\|_{p,q}+\|\bar{j}\|_{p,q}\|[j]\|_{p,q})\|_{L^{\frac{2p-3+q}{p}}(\mathbb{R}^{+})}\vspace{1ex}\\
&\qquad\qquad\leq C_8C_9(I_{p,q}([\omega])I_{p,q}(\omega)+I_{p,q}([j])I_{p,q}(j)\vspace{2ex}\\
&\qquad\qquad\quad+I_{p,q}(\bar{\omega})I_{p,q}([\omega])+I_{p,q}(\bar{j})I_{p,q}([j])),
\end{array}\eqno(3.54)$$
$$I_{p,q}([j])\leq\displaystyle C_8C_9(I_{p,q}([\omega])(I_{p,q}(j)+I_{p,q}(\bar{j}))
+I_{p,q}([\omega])I_{p,q}(\bar{j})+I_{p,q}(\bar{\omega})I_{p,q}([j])).\eqno(3.55)$$
Hence, (3.53)-(3.55) and (3.45) yield that
$$I_{p,q}([\omega])+I_{p,q}([j])\leq C_{10}(\|\omega_0\|_{p_0,q_0}+\|j_0\|_{p_0,q_0})(I_{p,q}([\omega])+I_{p,q}([j]))\eqno(3.56)$$
whenever (3.13) holds.

Letting
$$\|\omega_0\|_{p_0,q_0}+\|j_0\|_{p_0,q_0}<\min\{G_1,(2C_9)^{-1}\},$$
inequality (3.56) gives $I_{p,q}([\omega])+I_{p,q}([j])=0$, which 
implies that $[\omega]=0$ and $[j]=0$ for almost every 
$(x,t)\in\mathbb{R}^3\times\mathbb{R}^{+}$. This proves the uniqueness. $\hfill\Box$

\medskip

\quad\hspace{-20pt}{\bf 3.2. Proof for part (ii) of Theorem \ref{thm3.1}}

\medskip

Let $(p,q,p_0,q_0,\tilde{q_0},q_1)\in E_2$. Then there exist
$q_2,\,q_3\in[0,3)$, $\tilde{p}\in(1,\min\{p,p'\})$ and $\theta\in(0,1)$
such that
\begin{equation*}
\begin{array}{ll}
&0\leq q_1-\tilde{q_0}<1,\ 0\leq q_1-q_2<1,\vspace{1ex}\\
&q_2=\frac{q_3}{p'}+\frac{q}{p},\ \ \frac{1}{\tilde{p}}=\frac{1}{p'}+\frac{1}{3-q_3},\ \frac{1}{\tilde{p}}=\theta+\frac{1-\theta}{p},\vspace{1ex}\\
&\frac{q_3}{\tilde{p}}=q_1\theta+\frac{q}{p}(1-\theta),\frac{q_2-\tilde{q_0}+1}{2}=\frac{q_1-\tilde{q_0}}{2}\theta+\frac{2p-3+q}{2p}(2-\theta).
\end{array}
\end{equation*}
For convenience, we set
\begin{equation*}
\begin{array}{ll}
&W_{k,1,q_1}^0=\sup\limits_{t\in\mathbb{R}^{+}}t^{\frac{q_1-\tilde{q_0}}{2}}\|\omega^{(k)}(\cdot,t)\|_{1,q_1},
\ \ \ \ \ \ \ \ \ \ \bar{W}_{k,1,q_1}^0=\|\|\omega^{(k)}(\cdot,t)\|_{1,q_1}\|_{L_t^{\frac{2}{q_1-\tilde{q_0}}}(\mathbb{R}^{+})},\\
&W_{k,1,q_1}^1=\sup\limits_{t\in\mathbb{R}^{+}}t^{\frac{1+q_1-\tilde{q_0}}{2}}\|\nabla \omega^{(k)}(\cdot,t)\|_{1,q_1},
\ \ \ \ \ \bar{W}_{k,1,q_1}^1=\|\|\nabla \omega^{(k)}(\cdot,t)\|_{1,q_1}\|_{L_t^{\frac{2}{1+q_1-\tilde{q_0}}}(\mathbb{R}^{+})},\\
&J_{k,1,q_1}^0=\sup\limits_{t\in\mathbb{R}^{+}}t^{\frac{q_1-\tilde{q_0}}{2}}\|j^{(k)}(\cdot,t)\|_{1,q_1},
\ \ \ \ \ \ \ \ \ \ \ \ \ \bar{J}_{k,1,q_1}^0=\|\|j^{(k)}(\cdot,t)\|_{1,q_1}\|_{L_t^{\frac{2}{q_1-\tilde{q_0}}}(\mathbb{R}^{+})},\\
&J_{k,1,q_1}^1=\sup\limits_{t\in\mathbb{R}^{+}}t^{\frac{1+q_1-\tilde{q_0}}{2}}\|\nabla j^{(k)}(\cdot,t)\|_{1,q_1},
\ \ \ \ \ \ \ \bar{J}_{k,1,q_1}^1=\|\|\nabla j^{(k)}(\cdot,t)\|_{1,q_1}\|_{L_t^{\frac{2}{1+q_1-\tilde{q_0}}}(\mathbb{R}^{+})}.
\end{array}
\end{equation*}

{\bf Step 1: Estimates for  the terms $W_{k,1,q_1}^0$, $J_{k,1,q_1}^0$}.

\medskip

By (3.3) and Proposition 4, one finds that
$$\begin{array}{ll}
&\|\omega^{(k+1)}(\cdot,t)\|_{1,q_1}\leq\displaystyle\|G(\cdot,t)\ast \omega_0\|_{1,q_1}+B_1\int_0^t(t-s)^{-\frac{1}{2}-\frac{q_1-q_2}{2}}\\
&\qquad\qquad\qquad\qquad\times\|(u^{i,(k)}\omega^{(k)}-u^{(k)}\omega^{i,(k)}-b^{i,(k)}j^{(k)}+b^{(k)}j^{i,(k)})(\cdot,s)\|_{1,q_2}{\rm d}s,
\end{array}\eqno(3.57)$$

$$\begin{array}{ll}
&\|j^{(k+1)}(\cdot,t)\|_{1,q_1}\leq\|G(\cdot,t)\ast j_0\|_{1,q_1}\vspace{1ex}\\
&\qquad\qquad\qquad\qquad+\displaystyle B_1\int_0^t(t-s)^{-\frac{1}{2}-\frac{q_1-q_2}{2}}\|((u^{(k)}\cdot\nabla)b^{(k)}-(b^{(k)}\cdot\nabla)u^{(k)})(\cdot,s)\|_{1,q_2}{\rm d}s.
\end{array}\eqno(3.58)$$
Note that $q_2=\frac{q_3}{p'}+\frac{q}{p}$. By Proposition 2 (i),
one has
$$\begin{array}{ll}
&\quad\|(u^{i,(k)}\omega^{(k)}-u^{(k)}\omega^{i,(k)}-b^{i,(k)}j^{(k)}+b^{(k)}j^{i,(k)})(\cdot,s)\|_{1,q_2}\vspace{1ex}\\
&\leq2\|u^{(k)}(\cdot,s)\|_{p',q_3}\|\omega^{(k)}(\cdot,s)\|_{p,q}+2\|b^{(k)}(\cdot,s)\|_{p',q_3}\|j^{(k)}(\cdot,s)\|_{p,q},
\end{array}\eqno(3.59)$$
$$\begin{array}{ll}
&\quad\|((u^{(k)}\cdot\nabla)b^{(k)}-(b^{(k)}\cdot\nabla)u^{(k)})(\cdot,s)\|_{1,q_2}\vspace{1ex}\\
&\leq\|u^{(k)}(\cdot,s)\|_{p',q_3}\|j^{(k)}(\cdot,s)\|_{p,q}+\|b^{(k)}(\cdot,s)\|_{p',q_3}\|\omega^{(k)}(\cdot,s)\|_{p,q}.
\end{array}\eqno(3.60)$$
Note that $\tilde{p}<p'$ and $\frac{1}{\tilde{p}}-\frac{1}{p'}=\frac{1}{3-q_3}$.
This together with Proposition 3 implies that
$$\|u^{(k)}(\cdot,s)\|_{p',q_3}\leq B_2\|\omega^{(k)}(\cdot,s)\|_{\tilde{p},1},\eqno(3.61)$$
$$\|b^{(k)}(\cdot,s)\|_{p',q_3}\leq B_2\|j^{(k)}(\cdot,s)\|_{\tilde{p},1}.\eqno(3.62)$$
Since
$$1<\tilde{p}<\min\{p,p'\},\ \ \frac{1}{\tilde{p}}=\theta+\frac{1-\theta}{p},\ \ \frac{q_3}{\tilde{p}}=q_1\theta+\frac{q}{p}(1-\theta),$$
then by invoking (vi) of Proposition 1, we can get
$$\|\omega^{(k)}(\cdot,s)\|_{\tilde{p},q_3}\leq B_3\|\omega^{(k)}(\cdot,s)\|_{1,q_1}^\theta\|\omega^{(k)}(\cdot,s)\|_{p,q}^{1-\theta},\eqno(3.63)$$
$$\|j^{(k)}(\cdot,s)\|_{\tilde{p},q_3}\leq B_3\|j^{(k)}(\cdot,s)\|_{1,q_1}^\theta\|j^{(k)}(\cdot,s)\|_{p,q}^{1-\theta}.\eqno(3.64)$$
Combing (3.59) with (3.61)-(3.64) implies that
$$\begin{array}{ll}
&\quad\|(u^{i,(k)}\omega^{(k)}-u^{(k)}\omega^{i,(k)}-b^{i,(k)}j^{(k)}+b^{(k)}j^{i,(k)})(\cdot,s)\|_{1,q_2}\vspace{1ex}\\
&\leq2B_2\|\omega^{(k)}(\cdot,s)\|_{\tilde{p},q_3}\|\omega^{(k)}(\cdot,s)\|_{p,q}+2B_2\|j^{(k)}(\cdot,s)\|_{\tilde{p},q_3}\|j^{(k)}(\cdot,s)\|_{p,q}\vspace{1ex}\\
&\leq2B_2B_3(\|\omega^{(k)}(\cdot,s)\|_{1,q_1}^\theta\|\omega^{(k)}(\cdot,s)\|_{p,q}^{2-\theta}+\|j^{(k)}(\cdot,s)\|_{1,q_1}^\theta\|j^{(k)}(\cdot,s)\|_{p,q}^{2-\theta}).
\end{array}\eqno(3.65)$$
It follows from (3.60)-(3.64) that
$$\begin{array}{ll}
&\quad\|((u^{(k)}\cdot\nabla)b^{(k)}-(b^{(k)}\cdot\nabla)u^{(k)})(\cdot,s)\|_{1,q_2}\vspace{1ex}\\
&\leq B_2\|\omega^{(k)}(\cdot,s)\|_{\tilde{p},q_3}\|j^{(k)}(\cdot,s)\|_{p,q}+B_2\|j^{(k)}(\cdot,s)\|_{\tilde{p},q_3}\|\omega^{(k)}(\cdot,s)\|_{p,q}\vspace{1ex}\\
&\leq B_2B_3(\|\omega^{(k)}(\cdot,s)\|_{1,q_1}^\theta\|\omega^{(k)}(\cdot,s)\|_{p,q}^{1-\theta}\|j^{(k)}(\cdot,s)\|_{p,q}\vspace{1ex}\\
&\quad+\|j^{(k)}(\cdot,s)\|_{1,q_1}^\theta\|j^{(k)}(\cdot,s)\|_{p,q}^{1-\theta}\|\omega^{(k)}(\cdot,s)\|_{p,q}).
\end{array}\eqno(3.66)$$
Inequalities (3.57) and (3.65) imply that
$$\begin{array}{ll}
&\|\omega^{(k+1)}(\cdot,t)\|_{1,q_1}
\leq\|G(\cdot,t)\ast \omega_0\|_{1,q_1}\vspace{1ex}\\
&\qquad\qquad\qquad\qquad+\displaystyle 2B_1B_2B_3\int_0^t(t-s)^{-\frac{1}{2}-\frac{q_1-q_2}{2}}\|\omega^{(k)}(\cdot,s)\|_{1,q_1}^\theta\|\omega^{(k)}(\cdot,s)\|_{p,q}^{2-\theta}{\rm d}s\\
&\qquad\qquad\qquad\qquad+\displaystyle 2B_1B_2B_3\int_0^t(t-s)^{-\frac{1}{2}-\frac{q_1-q_2}{2}}\|j^{(k)}(\cdot,s)\|_{1,q_1}^\theta\|j^{(k)}(\cdot,s)\|_{p,q}^{2-\theta}{\rm d}s.
\end{array}\eqno(3.67)$$
Using (3.58) and (3.66), one has
$$\begin{array}{ll}
&\quad\|j^{(k+1)}(\cdot,t)\|_{1,q_1}\vspace{1ex}\\
&\leq\|G(\cdot,t)\ast j_0\|_{1,q_1}\vspace{1ex}\\
&\quad+\displaystyle B_1B_2B_3\int_0^t(t-s)^{-\frac{1}{2}-\frac{q_1-q_2}{2}}\|\omega^{(k)}(\cdot,s)\|_{1,q_1}^\theta\|\omega^{(k)}(\cdot,s)\|_{p,q}^{1-\theta}\|j^{(k)}(\cdot,s)\|_{p,q}{\rm d}s\\
&\quad+\displaystyle B_1B_2B_3\int_0^t(t-s)^{-\frac{1}{2}-\frac{q_1-q_2}{2}}\|j^{(k)}(\cdot,s)\|_{1,q_1}^\theta\|j^{(k)}(\cdot,s)\|_{p,q}^{1-\theta}\|\omega^{(k)}(\cdot,s)\|_{p,q}{\rm d}s.
\end{array}\eqno(3.68)$$
Invoking Proposition 4, (3.67) and (3.68) yield that
$$\begin{array}{ll}
&W_{k+1,1,q_1}^0\leq A_5\|\omega_0\|_{1,\tilde{q_0}}+2B_1B_2B_3(W_{k,1,q_1}^0)^\theta(W_{k,p,q}^0)^{2-\theta}\vspace{1ex}\\
&\qquad\qquad\times\displaystyle t^{\frac{q_1-\tilde{q_0}}{2}}\int_0^t(t-s)^{-\frac{1}{2}-\frac{q_1-q_2}{2}}s^{-\frac{q_1-\tilde{q_0}}{2}\theta}s^{(\frac{3-q}{2p}-1)(2-\theta)}{\rm d}s\vspace{1ex}\\
&\qquad\qquad+2B_1B_2B_3(J_{k,1,q_1}^0)^\theta (J_{k,p,q}^0)^{2-\theta}\vspace{1ex}\\
&\qquad\qquad\times\displaystyle t^{\frac{q_1-\tilde{q_0}}{2}}\int_0^t(t-s)^{-\frac{1}{2}-\frac{q_1-q_2}{2}}s^{-\frac{q_1-\tilde{q_0}}{2}\theta}s^{(\frac{3-q}{2p}-1)(2-\theta)}{\rm d}s,
\end{array}\eqno(3.69)$$
$$\begin{array}{ll}
&J_{k+1,1,q_1}^0\leq A_5\|j_0\|_{1,\tilde{q_0}}+\displaystyle B_1B_2B_3(W_{k,1,q_1}^0)^\theta(W_{k,p,q}^0)^{1-\theta}J_{k,p,q}^0\vspace{1ex}\\
&\qquad\qquad\times\displaystyle\int_0^t(t-s)^{-\frac{1}{2}-\frac{q_1-q_2}{2}}s^{-\frac{q_1-\tilde{q_0}}{2}\theta}s^{(\frac{3-q}{2p}-1)(2-\theta)}{\rm d}s\vspace{1ex}\\
&\qquad\qquad+\displaystyle B_1B_2B_3(J_{k,1,q_1}^0)^\theta(J_{k,p,q}^0)^{1-\theta}W_{k,p,q}^0\vspace{1ex}\\
&\qquad\qquad\times\displaystyle\int_0^t(t-s)^{-\frac{1}{2}-\frac{q_1-q_2}{2}}s^{-\frac{q_1-\tilde{q_0}}{2}\theta}s^{(\frac{3-q}{2p}-1)(2-\theta)}{\rm d}s.
\end{array}\eqno(3.70)$$
Observing that
$$\frac{1}{2}-\frac{q_1-q_2}{2}>0,\ \ \ \Big(\frac{3-q}{2p}-1\Big)(2-\theta)-\frac{q_1-\tilde{q_0}}{2}\theta+1>0,$$
$$\frac{1}{2}-\frac{q_1-q_2}{2}+\Big(\frac{3-q}{2p}-1\Big)(2-\theta)-\frac{q_1-\tilde{q_0}}{2}\theta=\frac{\tilde{q_0}-q_1}{2}.$$
Thus, it holds that
$$\int_0^t(t-s)^{-\frac{1}{2}-\frac{q_1-q_2}{2}}s^{-\frac{q_1-\tilde{q_0}}{2}\theta}s^{(\frac{3-q}{2p}-1)(2-\theta)}{\rm d}s\leq Ct^{\frac{\tilde{q_0}-q_1}{2}}.\eqno(3.71)$$
It follows from (3.69)-(3.71) that
$$W_{k+1,1,q_1}^0\leq A_5\|\omega_0\|_{1,\tilde{q_0}}+B_4((W_{k,1,q_1}^0)^\theta(W_{k,p,q}^0)^{2-\theta})+B_4((J_{k,1,q_1}^0)^\theta(J_{k,p,q}^0)^{2-\theta}),\eqno(3.72)$$
$$J_{k+1,1,q_1}^0\leq A_5\|j_0\|_{1,\tilde{q_0}}+B_4((W_{k,1,q_1}^0)^\theta(W_{k,p,q}^0)^{1-\theta}J_{k,p,q}^0)
+B_4((J_{k,1,q_1}^0)^\theta(J_{k,p,q}^0)^{1-\theta}W_{k,p,q}^0).\eqno(3.73)$$
Inequality (3.12) together with (3.72) and (3.73) may lead to
$$\begin{array}{ll}
&\quad W_{k+1,1,q_1}^0+J_{k+1,1,q_1}^0\vspace{1ex}\\
&\leq A_5(\|\omega_0\|_{1,\tilde{q_0}}+\|j_0\|_{1,\tilde{q_0}})\vspace{1ex}\\
&\quad +B_4((W_{k,1,q_1}^0)^\theta(W_{k,p,q}^0)^{2-\theta}+(J_{k,1,q_1}^0)^\theta(J_{k,p,q}^0)^{2-\theta})\vspace{1ex}\\
&\quad+B_4((W_{k,1,q_1}^0)^\theta(W_{k,p,q}^0)^{1-\theta}J_{k,p,q}^0)+B_4((J_{k,1,q_1}^0)^\theta(J_{k,p,q}^0)^{1-\theta}W_{k,p,q}^0)\vspace{1ex}\\
&\leq A_5(\|\omega_0\|_{1,\tilde{q_0}}+\|j_0\|_{1,\tilde{q_0}})
+B_5(\|\omega_0\|_{p_0,q_0}+\|j_0\|_{p_0,q_0})^{2-\theta}(W_{k,1,q_1}^0+J_{k,1,q_1}^0)^\theta,
\end{array}\eqno(3.74)$$
whenever (3.13) holds.

On the other hand, by Proposition 4, it holds that
$$W_{1,1,q_1}^0+J_{1,1,q_1}^0\leq A_5(\|\omega_0\|_{1,\tilde{q_0}}+\|j_0\|_{1,\tilde{q_0}}).\eqno(3.75)$$
Assume that (3.13) holds and
$$\max\{A_5(\|\omega_0\|_{1,\tilde{q_0}}+\|j_0\|_{1,\tilde{q_0}}),B_5(\|\omega_0\|_{p_0,q_0}+\|j_0\|_{p_0,q_0})^{2-\theta}\}\leq1/2.\eqno(3.76)$$
Then (3.74) and (3.75) may yields that
$$W_{k,1,q_1}^0+J_{k,1,q_1}^0\leq2\ \ \ {\rm for\ all}\  k\geq1,\eqno(3.77)$$
whenever (3.13) and (3.75) hold. (3.77) together with (3.74) implies that
$$W_{k,1,q_1}^0+J_{k,1,q_1}^0\leq C_{10}(\|\omega_0\|_{1,\tilde{q_0}}+\|j_0\|_{1,\tilde{q_0}}+\|\omega_0\|_{p_0,q_0}+\|j_0\|_{p_0,q_0})\ \ \ \ {\rm for\ all}\ k\geq1\eqno(3.78)$$
whenever (3.13) and (3.76) hold.
\medskip

{\bf Step 2: Estimates for  the terms $\bar{W}_{k,1,q_1}^0$ and $\bar{J}_{k,1,q_1}^0$.}

\medskip

Note that
$$\frac{2}{q_1-\tilde{q_0}}>1,\ \frac{2}{1+q_1-q_2}>1,\ \frac{2}{q_2-\tilde{q_0}+1}>1,\ \frac{q_1-\tilde{q_0}}{2}+1
=\frac{1+q_1-q_2}{2}+\frac{q_2-\tilde{q_0}+1}{2}.$$
By (3.67), (3.68), Proposition 4, Lemma 2 and the following Young's
inequality
$$\|f*g\|_{L^{\frac{2}{q_1-\tilde{q_0}}}(\mathbb{R}^{+})}\leq
\|f\|_{L^{\frac{2}{1+q_1-q_2},\infty}(\mathbb{R}^{+})}\|g\|_{L^{\frac{2}{q_2-\tilde{q_0}+1}}(\mathbb{R}^{+})},$$
one can obtain that
$$\begin{array}{ll}
&\bar{W}_{k+1,1,q_1}^0
\leq A_5\|\omega_0\|_{1,\tilde{q_0}}+2B_1B_2B_3\|t^{-\frac{1+q_1-q_2}{2}}\|_{L^{\frac{2}{1+q_1-q_2},\infty}(\mathbb{R}^{+})}\vspace{1ex}\\
&\qquad\qquad\quad\times\|(\|\omega^{(k)}(\cdot,t)\|_{1,q_1}^\theta\|\omega^{(k)}(\cdot,t)\|_{p,q}^{2-\theta}\vspace{1ex}\\
&\qquad\qquad\quad+\|j^{(k)}(\cdot,t)\|_{1,q_1}^\theta\|j^{(k)}(\cdot,t)\|_{p,q}^{2-\theta})\|_{L^{\frac{2}{q_2-\tilde{q_0}+1}}(\mathbb{R}^{+})},
\end{array}\eqno(3.79)$$
$$\begin{array}{ll}
&\bar{J}_{k+1,1,q_1}^0
\leq A_5\|j_0\|_{1,\tilde{q_0}}+B_1B_2B_3\|t^{-\frac{1+q_1-q_2}{2}}\|_{L^{\frac{2}{1+q_1-q_2},\infty}(\mathbb{R}^{+})}\vspace{1ex}\\
&\qquad\qquad\quad\times\|(\|\omega^{(k)}(\cdot,t)\|_{1,q_1}^\theta\|\omega^{(k)}(\cdot,t)\|_{p,q}^{1-\theta}\|j^{(k)}(\cdot,t)\|_{p,q}\vspace{1ex}\\
&\qquad\qquad\quad+\|j^{(k)}(\cdot,t)\|_{1,q_1}^\theta\|j^{(k)}(\cdot,t)\|_{p,q}^{1-\theta}\|\omega^{(k)}(\cdot,t)\|_{p,q})\|_{L^{\frac{2}{q_2-\tilde{q_0}+1}}(\mathbb{R}^{+})}.
\end{array}\eqno(3.80)$$
Observe that
$$\frac{2}{q_1-\tilde{q_0}}>\theta,\ \ \ \frac{2p}{2p-3+q}>2-\theta,\ \ \ \frac{q_2-\tilde{q_0}+1}{2}=\frac{q_1-\tilde{q_0}}{2}\theta+\frac{2p-3+q}{2p}(2-\theta).$$
(3.79) and (3.80) together with H\"{o}lder's inequality imply that
$$\bar{W}_{k+1,1,q_1}^0\leq A_5\|\omega_0\|_{1,\tilde{q_0}}+B_6((\bar{W}_{k,1,q_1}^0)^\theta(\bar{W}_{k,p,q}^0)^{2-\theta}
+(\bar{J}_{k,1,q_1}^0)^\theta(\bar{J}_{k,p,q}^0)^{2-\theta}),\eqno(3.81)$$
$$\bar{J}_{k+1,1,q_1}^0\leq A_5\|j_0\|_{1,\tilde{q_0}}+B_6((\bar{W}_{k,1,q_1}^0)^\theta(\bar{W}_{k,p,q}^0)^{1-\theta}\bar{J}_{k,p,q}^0
+(\bar{J}_{k,1,q_1}^0)^\theta(\bar{J}_{k,p,q}^0)^{1-\theta}\bar{W}_{k,p,q}^0).\eqno(3.82)$$
By (3.81), (3.82) and (3.12), we have
$$\begin{array}{ll}
\bar{W}_{k+1,1,q_1}^0+\bar{J}_{k+1,1,q_1}^0&\leq A_5(\|\omega_0\|_{1,\tilde{q_0}}+\|j_0\|_{1,\tilde{q_0}})\vspace{1ex}\\
&\quad+16A_1B_6(\bar{W}_{k,1,q_1}^0+\bar{J}_{k,1,q_1}^0)^\theta(\|\omega_0\|_{p_0,q_0}+\|j_0\|_{p_0,q_0})^{2-\theta},
\end{array}\eqno(3.83)$$
whenever (3.13) holds.

On the other hand, by Proposition 4 and Lemma 2, it holds that
$$\bar{W}_{1,1,q_1}^0+\bar{J}_{1,1,q_1}^0\leq A_5(\|\omega_0\|_{1,\tilde{q_0}}+\|j_0\|_{1,\tilde{q_0}}).\eqno(3.84)$$
By (3.83), (3.84) and the argument similar to those used in
deriving (3.78), it holds that
$$\bar{W}_{k,1,q_1}^0+\bar{J}_{k,1,q_1}^0\leq 2\ \ \ \ {\rm for\ all}\ k\geq1\eqno(3.85)$$
whenever (3.13) holds and
$$\max\{A_5(\|\omega_0\|_{1,\tilde{q_0}}+\|j_0\|_{1,\tilde{q_0}}),16A_1B_6(\|\omega_0\|_{p_0,q_0}+\|j_0\|_{p_0,q_0})^{2-\theta}\}\leq1/2.\eqno(3.86)$$
Inequality (3.85) together with (3.83) also yields that
$$\bar{W}_{k,1,q_1}^0+\bar{J}_{k,1,q_1}^0\leq C_{11}(\|\omega_0\|_{1,\tilde{q_0}}+\|j_0\|_{1,\tilde{q_0}}+\|\omega_0\|_{p_0,q_0}+\|j_0\|_{p_0,q_0})\ \ \ \ {\rm for\ all}\ k\geq1\eqno(3.87)$$
whenever (3.13) and (3.86) hold.

{\bf Step 3: Estimates for  the terms $W_{k,1,q_1}^1,\,J_{k,1,q_1}^1$.}

By (3.3) and Proposition 4, we get
$$\begin{array}{ll}
&\quad\|\nabla\omega^{(k+1)}\|_{1,q_1}\\
&\leq\displaystyle\|\nabla G(\cdot,t)\ast\omega_0\|_{1,q_1}+\displaystyle B_3\int_{t/2}^t(t-s)^{-\frac{1}{2}-\frac{q_1-q_2}{2}}\|((u^{(k)}\cdot\nabla)\omega^{(k)}-(\omega^{(k)}\cdot\nabla)u^{(k)}\\
&\quad-(b^{(k)}\cdot\nabla)j^{(k)}+(j^{(k)}\cdot\nabla)b^{(k)})(\cdot,s)\|_{1,q_2}{\rm d}s\\
&\quad+\displaystyle B_3\int_{0}^{t/2}(t-s)^{-1-\frac{q_1-q_2}{2}}\|(u^{i,(k)}\omega^{(k)}-u^{(k)}\omega^{i,(k)}\\
&\quad-b^{i,(k)}j^{(k)}+b^{(k)}j^{i,(k)})(\cdot,s)\|_{1,q_2}{\rm d}s,
\end{array}\eqno(3.88)$$
$$\begin{array}{ll}
&\quad\|\nabla j^{(k+1)}\|_{1,q_1}\\
&\leq\displaystyle\|\nabla G(\cdot,t)\ast j_0\|_{1,q_1}+\displaystyle B_3\int_{t/2}^t(t-s)^{-\frac{1}{2}-\frac{q_1-q_2}{2}}\|((u^{(k)}\cdot\nabla)j^{(k)}+\nabla u^{i,(k)}\times b^{(k)}_{x_i}\\
&\quad-(b^{(k)}\cdot\nabla)\omega^{(k)}-\nabla b^{i,(k)}\times u^{(k)}_{x_i})(\cdot,s)\|_{1,q_2}{\rm d}s\\
&\quad+\displaystyle B_3\int_0^{t/2}(t-s)^{-1-\frac{q_1-q_2}{2}}\|((u^{(k)}\cdot\nabla)b^{(k)}-(b^{(k)}\cdot\nabla)u^{(k)})(\cdot,s)\|_{1,q_2}{\rm d}s.
\end{array}\eqno(3.89)$$
By the arguments similar to those used in deriving (3.65) and (3.66),
one can get
$$\begin{array}{ll}
&\quad\|((u^{(k)}\cdot\nabla)\omega^{(k)}-(\omega^{(k)}\cdot\nabla)u^{(k)}-(b^{(k)}\cdot\nabla)j^{(k)}+(j^{(k)}\cdot\nabla)b^{(k)})(\cdot,s)\|_{1,q_2}\vspace{1ex}\\
&\leq B_4(\|\omega^{(k)}(\cdot,s)\|_{\tilde{p},1}\|\nabla\omega^{(k)}(\cdot,s)\|_{p,q}+\|\nabla\omega^{(k)}(\cdot,s)\|_{\tilde{p},1}\|\omega^{(k)}(\cdot,s)\|_{p,q})\vspace{1ex}\\
&\quad+B_4(\|\nabla j^{(k)}(\cdot,s)\|_{\tilde{p},1}\|j^{(k)}(\cdot,s)\|_{p,q}+\|j^{(k)}(\cdot,s)\|_{\tilde{p},1}\|\nabla j^{(k)}(\cdot,s)\|_{p,q})\vspace{1ex}\\
&\leq B_4(\|\omega^{(k)}(\cdot,s)\|_{1,q_1}^\theta\|\omega^{(k)}(\cdot,s)\|_{p,q}^{(1-\theta)}\|\nabla\omega^{(k)}(\cdot,s)\|_{p,q}\vspace{1ex}\\
&\quad+\|\nabla\omega^{(k)}(\cdot,s)\|_{1,q_1}^\theta\|\nabla\omega^{(k)}(\cdot,s)\|_{p,q}^{(1-\theta)}\|\omega^{(k)}(\cdot,s)\|_{p,q})\vspace{1ex}\\
&\quad+B_4(\|\nabla j^{(k)}(\cdot,s)\|_{1,q_1}^\theta\|\nabla j^{(k)}(\cdot,s)\|_{p,q}^{(1-\theta)}\|j^{(k)}(\cdot,s)\|_{p,q}\vspace{1ex}\\
&\quad+\|j^{(k)}(\cdot,s)\|_{1,q_1}^\theta\|j^{(k)}(\cdot,s)\|_{p,q}^{(1-\theta)}\|\nabla j^{(k)}(\cdot,s)\|_{p,q}),
\end{array}\eqno(3.90)$$
$$\begin{array}{ll}
&\quad \|((u^{(k)}\cdot\nabla)j^{(k)}+\nabla u^{i,(k)}\times b^{(k)}_{x_i}-(b^{(k)}\cdot\nabla)\omega^{(k)}-\nabla b^{i,(k)}\times u^{(k)}_{x_i})(\cdot,s)\|_{1,q_2}\vspace{1ex}\\
&\leq B_4(\|\omega^{(k)}(\cdot,s)\|_{\tilde{p},1}\|\nabla j^{(k)}(\cdot,s)\|_{p,q}+2\|\omega^{(k)}(\cdot,s)\|_{p,q}\|\nabla j^{(k)}(\cdot,s)\|_{\tilde{p},1}\vspace{1ex}\\
&\quad+\|j^{(k)}(\cdot,s)\|_{\tilde{p},1}\|\nabla\omega^{(k)}(\cdot,s)\|_{p,q})\vspace{1ex}\\
&\leq B_4(\|\omega^{(k)}(\cdot,s)\|_{1,q_1}^\theta\|\omega^{(k)}(\cdot,s)\|_{p,q}^{(1-\theta)}\|\nabla j^{(k)}(\cdot,s)\|_{p,q}\vspace{1ex}\\
&\quad+2\|\omega^{(k)}(\cdot,s)\|_{p,q}\|\nabla j^{(k)}(\cdot,s)\|_{1,q_1}^\theta\|\nabla j^{(k)}(\cdot,s)\|_{p,q}^{(1-\theta)}\vspace{1ex}\\
&\quad+\|j^{(k)}(\cdot,s)\|_{1,q_1}^\theta\|j^{(k)}(\cdot,s)\|_{p,q}^{(1-\theta)}\|\nabla\omega^{(k)}(\cdot,s)\|_{p,q}).
\end{array}\eqno(3.91)$$
It follows from (3.90) that
$$\begin{array}{ll}
&\quad\displaystyle\int_{t/2}^t(t-s)^{-\frac{1}{2}-\frac{q_1-q_2}{2}}\|((u^{(k)}\cdot\nabla)\omega^{(k)}-(\omega^{(k)}\cdot\nabla)u^{(k)}\\
&\quad-(b^{(k)}\cdot\nabla)j^{(k)}+(j^{(k)}\cdot\nabla)b^{(k)})(\cdot,s)\|_{1,q_2}{\rm d}s\\
&\leq\displaystyle B_4\int_{t/2}^t(t-s)^{-\frac{1}{2}-\frac{q_1-q_2}{2}}(\|\omega^{(k)}(\cdot,s)\|_{1,q_1}^\theta\|\omega^{(k)}(\cdot,s)\|_{p,q}^{(1-\theta)}\|\nabla\omega^{(k)}(\cdot,s)\|_{p,q}){\rm d}s\\
&\quad+\displaystyle B_4\int_{t/2}^t(t-s)^{-\frac{1}{2}-\frac{q_1-q_2}{2}}(\|\nabla\omega^{(k)}(\cdot,s)\|_{1,q_1}^\theta\|\nabla\omega^{(k)}(\cdot,s)\|_{p,q}^{(1-\theta)}\|\omega^{(k)}(\cdot,s)\|_{p,q}){\rm d}s\\
&\quad+\displaystyle B_4\int_{t/2}^t(t-s)^{-\frac{1}{2}-\frac{q_1-q_2}{2}}(\|\nabla j^{(k)}(\cdot,s)\|_{1,q_1}^\theta\|\nabla j^{(k)}(\cdot,s)\|_{p,q}^{(1-\theta)}\|j^{(k)}(\cdot,s)\|_{p,q}){\rm d}s\\
&\quad+\displaystyle B_4\int_{t/2}^t(t-s)^{-\frac{1}{2}-\frac{q_1-q_2}{2}}(\|j^{(k)}(\cdot,s)\|_{1,q_1}^\theta\|j^{(k)}(\cdot,s)\|_{p,q}^{(1-\theta)}\|\nabla j^{(k)}(\cdot,s)\|_{p,q}){\rm d}s.
\end{array}\eqno(3.92)$$
Noting that
$$\frac{1}{2}-\frac{q_1-q_2}{2}>0,\ \ \ -\frac{q_1-\tilde{q_0}}{2}\theta-\frac{2p-3+q}{2p}(1-\theta)-\frac{3p-3+q}{2p}+1>0,$$
$$\frac{1}{2}-\frac{q_1-q_2}{2}-\frac{q_1-\tilde{q_0}}{2}\theta-\frac{2p-3+q}{2p}(1-\theta)-\frac{3p-3+q}{2p}=\frac{\tilde{q_0}-q_1-1}{2}.$$
It follows that
$$\begin{array}{ll}
&\quad\displaystyle\int_{t/2}^t(t-s)^{-\frac{1}{2}-\frac{q_1-q_2}{2}}(\|\omega^{(k)}(\cdot,s)\|_{1,q_1}^\theta\|\omega^{(k)}(\cdot,s)\|_{p,q}^{(1-\theta)}\|\nabla\omega^{(k)}(\cdot,s)\|_{p,q}){\rm d}s\vspace{1ex}\\
&\leq\displaystyle(W_{k,1,q_1}^0)^\theta(W_{k,p,q}^0)^{1-\theta}W_{k,p,q}^1\vspace{1ex}\\
&\quad\times\displaystyle\int_{t/2}^t(t-s)^{-\frac{1}{2}-\frac{q_1-q_2}{2}}s^{-\frac{q_1-\tilde{q_0}}{2}\theta}s^{-\frac{2p-3+q}{2p}(1-\theta)}s^{-\frac{3p-3+q}{2p}}{\rm d}s\vspace{1ex}\\
&\leq B_5(W_{k,1,q_1}^0)^\theta(W_{k,p,q}^0)^{1-\theta}W_{k,p,q}^1t^{\frac{\tilde{q_0}-q_1-1}{2}}.
\end{array}\eqno(3.93)$$
Similarly, we can get
$$\begin{array}{ll}
&\quad\displaystyle\int_{t/2}^t(t-s)^{-\frac{1}{2}-\frac{q_1-q_2}{2}}(\|\nabla\omega^{(k)}(\cdot,s)\|_{1,q_1}^\theta\|\nabla\omega^{(k)}(\cdot,s)\|_{p,q}^{(1-\theta)}\|\omega^{(k)}(\cdot,s)\|_{p,q}){\rm d}s\\
&\leq\displaystyle B_5(W_{k,1,q_1}^1)^\theta(W_{k,p,q}^1)^{1-\theta}W_{k,p,q}^0t^{\frac{\tilde{q_0}-q_1-1}{2}},
\end{array}\eqno(3.94)$$
$$\begin{array}{ll}
&\quad\displaystyle\int_{t/2}^t(t-s)^{-\frac{1}{2}-\frac{q_1-q_2}{2}}(\|\nabla j^{(k)}(\cdot,s)\|_{1,q_1}^\theta\|\nabla j^{(k)}(\cdot,s)\|_{p,q}^{(1-\theta)}\|j^{(k)}(\cdot,s)\|_{p,q}){\rm d}s\\
&\leq\displaystyle B_5(J_{k,1,q_1}^1)^\theta(W_{k,p,q}^1)^{1-\theta}W_{k,p,q}^0t^{\frac{\tilde{q_0}-q_1-1}{2}}.
\end{array}\eqno(3.95)$$
$$\begin{array}{ll}
&\quad\displaystyle\int_{t/2}^t(t-s)^{-\frac{1}{2}-\frac{q_1-q_2}{2}}(\|j^{(k)}(\cdot,s)\|_{1,q_1}^\theta\|j^{(k)}(\cdot,s)\|_{p,q}^{(1-\theta)}\|\nabla j^{(k)}(\cdot,s)\|_{p,q}){\rm d}s\\
&\leq\displaystyle B_5(J_{k,1,q_1}^0)^\theta(J_{k,p,q}^0)^{1-\theta}J_{k,p,q}^1t^{\frac{\tilde{q_0}-q_1-1}{2}}.
\end{array}\eqno(3.96)$$
It follows from (3.92)-(3.96) that
$$\begin{array}{ll}
&\quad\displaystyle\int_{t/2}^t(t-s)^{-\frac{1}{2}-\frac{q_1-q_2}{2}}\|((u^{(k)}\cdot\nabla)\omega^{(k)}-(\omega^{(k)}\cdot\nabla)u^{(k)}\\
&\quad-(b^{(k)}\cdot\nabla)j^{(k)}+(j^{(k)}\cdot\nabla)b^{(k)})(\cdot,s)\|_{1,q_2}{\rm d}s\vspace{1ex}\\
&\leq B_4B_5((W_{k,1,q_1}^0)^\theta(W_{k,p,q}^0)^{1-\theta}W_{k,p,q}^1+(W_{k,1,q_1}^1)^\theta(W_{k,p,q}^1)^{1-\theta}W_{k,p,q}^0)t^{\frac{\tilde{q_0}-q_1-1}{2}}\vspace{1ex}\\
&\quad+B_4B_5((J_{k,1,q_1}^1)^\theta(W_{k,p,q}^1)^{1-\theta}W_{k,p,q}^0+(J_{k,1,q_1}^0)^\theta(J_{k,p,q}^0)^{1-\theta}J_{k,p,q}^1)t^{\frac{\tilde{q_0}-q_1-1}{2}}.
\end{array}\eqno(3.97)$$
On the other hand, by (3.65) we get
$$\begin{array}{ll}
&\quad\displaystyle\int_{0}^{t/2}(t-s)^{-1-\frac{q_1-q_2}{2}}\|(u^{i,(k)}\omega^{(k)}-u^{(k)}\omega^{i,(k)}-b^{i,(k)}j^{(k)}+b^{(k)}j^{i,(k)})(\cdot,s)\|_{1,q_2}{\rm d}s\\
&\leq\displaystyle2B_2B_3\int_{0}^{t/2}(t-s)^{-1-\frac{q_1-q_2}{2}}\|\omega^{(k)}(\cdot,s)\|_{1,q_1}^\theta\|\omega^{(k)}(\cdot,s)\|_{p,q}^{2-\theta}{\rm d}s\\
&\quad+\displaystyle 2B_2B_3\int_{0}^{t/2}(t-s)^{-1-\frac{q_1-q_2}{2}}\|j^{(k)}(\cdot,s)\|_{1,q_1}^\theta\|j^{(k)}(\cdot,s)\|_{p,q}^{2-\theta}{\rm d}s.
\end{array}\eqno(3.98)$$
Note that
$$-1-\frac{q_1-q_2}{2}-\frac{q_1-\tilde{q_0}}{2}\theta-\frac{2p-3+q}{2p}(2-\theta)+1=\frac{\tilde{q_0}-q_1-1}{2}.$$
It follows that
$$\begin{array}{ll}
&\quad\displaystyle\int_{0}^{t/2}(t-s)^{-1-\frac{q_1-q_2}{2}}\|\omega^{(k)}(\cdot,s)\|_{1,q_1}^\theta\|\omega^{(k)}(\cdot,s)\|_{p,q}^{2-\theta}{\rm d}s\\
&\leq\displaystyle(W_{k,1,q_1}^0)^\theta(W_{k,p,q}^0)^{2-\theta}\int_{0}^{t/2}(t-s)^{-1-\frac{q_1-q_2}{2}}s^{-\frac{q_1-\tilde{q_0}}{2}\theta}s^{-\frac{2p-3+q}{2p}(2-\theta)}ds\vspace{1ex}\\
&\leq B_6(W_{k,1,q_1}^0)^\theta(W_{k,p,q}^0)^{2-\theta}t^{\frac{\tilde{q_0}-q_1-1}{2}}.
\end{array}\eqno(3.99)$$
Similarly, we can get
$$\int_{0}^{t/2}(t-s)^{-1-\frac{q_1-q_2}{2}}\|j^{(k)}(\cdot,s)\|_{1,q_1}^\theta\|j^{(k)}(\cdot,s)\|_{p,q}^{2-\theta}{\rm d}s
\leq B_6(J_{k,1,q_1}^0)^\theta(J_{k,p,q}^0)^{2-\theta}t^{\frac{\tilde{q_0}-q_1-1}{2}}.\eqno(3.100)$$
It follows from (3.88), Proposition 4 and (3.92)-(3.100) that
$$\begin{array}{ll}
&\quad W_{k+1,1,q_1}^1\vspace{1ex}\\
&\leq B_7\|\omega_0\|_{1,\tilde{q_0}}+B_7(W_{k,1,q_1}^0)^\theta(W_{k,p,q}^0)^{1-\theta}W_{k,p,q}^1+B_7(J_{k,1,q_1}^0)^\theta(J_{k,p,q}^0)^{1-\theta}J_{k,p,q}^1\vspace{1ex}\\
&\quad+B_7(W_{k,1,q_1}^1)^\theta(W_{k,p,q}^1)^{1-\theta}W_{k,p,q}^0+B_7(J_{k,1,q_1}^1)^\theta(W_{k,p,q}^1)^{1-\theta}W_{k,p,q}^0\vspace{1ex}\\
&\quad+B_7(W_{k,1,q_1}^0)^\theta(W_{k,p,q}^0)^{2-\theta}+B_7(J_{k,1,q_1}^0)^\theta(J_{k,p,q}^0)^{2-\theta}.
\end{array}\eqno(3.101)$$
We get from (3.91) that
$$\begin{array}{ll}
&\quad\displaystyle\int_{t/2}^t(t-s)^{-\frac{1}{2}-\frac{q_1-q_2}{2}}\|((u^{(k)}\cdot\nabla)j^{(k)}+\nabla u^{i,(k)}\times b^{(k)}_{x_i}\vspace{1ex}\\
&\qquad-(b^{(k)}\cdot\nabla)\omega^{(k)}-\nabla b^{i,(k)}\times u^{(k)}_{x_i})(\cdot,s)\|_{1,q_2}{\rm d}s\vspace{1ex}\\
&\leq \displaystyle B_4\int_{t/2}^t(t-s)^{-\frac{1}{2}-\frac{q_1-q_2}{2}}(\|\omega^{(k)}(\cdot,s)\|_{1,q_1}^\theta\|\omega^{(k)}(\cdot,s)\|_{p,q}^{(1-\theta)}\|\nabla j^{(k)}(\cdot,s)\|_{p,q}){\rm d}s\vspace{1ex}\\
&\quad+\displaystyle B_4\int_{t/2}^t(t-s)^{-\frac{1}{2}-\frac{q_1-q_2}{2}}\|j^{(k)}(\cdot,s)\|_{1,q_1}^\theta\|j^{(k)}(\cdot,s)\|_{p,q}^{(1-\theta)}\|\nabla\omega^{(k)}(\cdot,s)\|_{p,q}{\rm d}s\vspace{1ex}\\
&\quad+\displaystyle2B_4\int_{t/2}^t(t-s)^{-\frac{1}{2}-\frac{q_1-q_2}{2}}\|\omega^{(k)}(\cdot,s)\|_{p,q}\|\nabla j^{(k)}(\cdot,s)\|_{1,q_1}^\theta\|\nabla j^{(k)}(\cdot,s)\|_{p,q}^{(1-\theta)}{\rm d}s.
\end{array}\eqno(3.102)$$
Similar arguments to those in deriving (3.93) may yield that
$$\begin{array}{ll}
&\quad\displaystyle\int_{t/2}^t(t-s)^{-\frac{1}{2}-\frac{q_1-q_2}{2}}(\|\omega^{(k)}(\cdot,s)\|_{1,q_1}^\theta\|\omega^{(k)}(\cdot,s)\|_{p,q}^{(1-\theta)}\|\nabla j^{(k)}(\cdot,s)\|_{p,q}){\rm d}s\vspace{1ex}\\
&\leq B_8(W_{k,1,q_1}^0)^\theta(W_{k,p,q}^0)^{(1-\theta)}J_{k,p,q}^1t^{\frac{\tilde{q_0}-q_1-1}{2}},
\end{array}\eqno(3.103)$$
$$\begin{array}{ll}
&\quad\displaystyle\int_{t/2}^t(t-s)^{-\frac{1}{2}-\frac{q_1-q_2}{2}}(\|j^{(k)}(\cdot,s)\|_{1,q_1}^\theta\|j^{(k)}(\cdot,s)\|_{p,q}^{(1-\theta)}\|\nabla\omega^{(k)}(\cdot,s)\|_{p,q}){\rm d}s\vspace{1ex}\\
&\leq B_8(J_{k,1,q_1}^0)^\theta(J_{k,p,q}^0)^{(1-\theta)}W_{k,p,q}^1t^{\frac{\tilde{q_0}-q_1-1}{2}}.
\end{array}\eqno(3.104)$$
Note that
$$-\frac{1}{2}-\frac{q_1-q_2}{2}+\frac{3-q}{2p}-1+\frac{\tilde{q_0}-q_1-1}{2}\theta+\frac{3-q-3p}{2p}(1-\theta)+1=\frac{\tilde{q_0}-q_1-1}{2}.$$
It follows that
$$\begin{array}{ll}
&\quad\displaystyle\int_{t/2}^t(t-s)^{-\frac{1}{2}-\frac{q_1-q_2}{2}}\|\omega^{(k)}(\cdot,s)\|_{p,q}\|\nabla j^{(k)}(\cdot,s)\|_{1,q_1}^\theta\|\nabla j^{(k)}(\cdot,s)\|_{p,q}^{(1-\theta)}{\rm d}s\vspace{1ex}\\
&\leq\displaystyle W_{k,p,q}^0(J_{k,1,q_1}^1)^{\theta}(J_{k,p,q}^1)^{1-\theta}\int_{t/2}^t(t-s)^{-\frac{1}{2}-\frac{q_1-q_2}{2}}s^{\frac{3-q}{2p}-1}s^{\frac{\tilde{q_0}-q_1-1}{2}\theta}s^{\frac{3-q-3p}{2p}(1-\theta)}ds\vspace{1ex}\\
&\leq B_9W_{k,p,q}^0(J_{k,1,q_1}^1)^{\theta}(J_{k,p,q}^1)^{1-\theta}t^{\frac{\tilde{q_0}-q_1-1}{2}}.
\end{array}\eqno(3.105)$$
It follows from (3.102)-(3.105) that
$$\begin{array}{ll}
&\quad\displaystyle\int_{t/2}^t(t-s)^{-\frac{1}{2}-\frac{q_1-q_2}{2}}\|((u^{(k)}\cdot\nabla)j^{(k)}+\nabla u^{i,(k)}\times b^{(k)}_{x_i}\\
&\quad-(b^{(k)}\cdot\nabla)\omega^{(k)}-\nabla b^{i,(k)}\times u^{(k)}_{x_i})(\cdot,s)\|_{1,q_2}{\rm d}s\vspace{1ex}\\
&\leq B_{10}t^{\frac{\tilde{q_0}-q_1-1}{2}}((W_{k,1,q_1}^0)^\theta(W_{k,p,q}^0)^{(1-\theta)}J_{k,p,q}^1+(J_{k,1,q_1}^0)^\theta(J_{k,p,q}^0)^{(1-\theta)}W_{k,p,q}^1\vspace{1ex}\\
&\quad+W_{k,p,q}^0(J_{k,1,q_1}^1)^{\theta}(J_{k,p,q}^1)^{1-\theta}).
\end{array}\eqno(3.106)$$
Note that
$$-1-\frac{q_1-q_2}{2}-\frac{q_1-\tilde{q_0}}{2}\theta+\frac{3-q-2p}{2p}(1-\theta)+\frac{3-q-2p}{2p}+1=\frac{\tilde{q_0}-q_1-1}{2}.$$
By (3.66), it holds that
$$\begin{array}{ll}
&\quad\displaystyle\int_0^{t/2}(t-s)^{-1-\frac{q_1-q_2}{2}}\|((u^{(k)}\cdot\nabla)b^{(k)}-(b^{(k)}\cdot\nabla)u^{(k)})(\cdot,s)\|_{1,q_2}{\rm d}s\\
&\leq\displaystyle B_2B_3\int_0^{t/2}(t-s)^{-1-\frac{q_1-q_2}{2}}\|\omega^{(k)}(\cdot,s)\|_{1,q_1}^\theta\|\omega^{(k)}(\cdot,s)\|_{p,q}^{1-\theta}\|j^{(k)}(\cdot,s)\|_{p,q}{\rm d}s\\
&\quad+\displaystyle B_2B_3\int_0^{t/2}(t-s)^{-1-\frac{q_1-q_2}{2}}\|j^{(k)}(\cdot,s)\|_{1,q_1}^\theta\|j^{(k)}(\cdot,s)\|_{p,q}^{1-\theta}\|\omega^{(k)}(\cdot,s)\|_{p,q}{\rm d}s\vspace{1ex}\\
&\leq\displaystyle B_2B_3((W_{k,1,q_1}^0)^\theta(W_{k,p,q}^0)^{1-\theta}J_{k,p,q}^0+(J_{k,1,q_1}^0)^\theta(J_{k,p,q}^0)^{1-\theta}W_{k,p,q}^0)\vspace{1ex}\\
&\quad\times\displaystyle\int_0^{t/2}(t-s)^{-1-\frac{q_1-q_2}{2}}s^{-\frac{q_1-\tilde{q_0}}{2}\theta}s^{\frac{3-q-2p}{2p}(1-\theta)}s^{\frac{3-q-2p}{2p}}ds\vspace{1ex}\\
&\leq\displaystyle B_{11}t^{\frac{\tilde{q_0}-q_1-1}{2}}((W_{k,1,q_1}^0)^\theta(W_{k,p,q}^0)^{1-\theta}J_{k,p,q}^0+(J_{k,1,q_1}^0)^\theta(J_{k,p,q}^0)^{1-\theta}W_{k,p,q}^0).
\end{array}\eqno(3.107)$$
Using (3.89), (3.106), (3.107) and Proposition 4, we have
$$\begin{array}{ll}
&J_{k+1,1,q_1}^1\leq B_{12}\|j_0\|_{1,\tilde{q_0}}+B_{12}W_{k,p,q}^0(J_{k,1,q_1}^1)^{\theta}(J_{k,p,q}^1)^{1-\theta}\vspace{1ex}\\
&\qquad\qquad+B_{12}((W_{k,1,q_1}^0)^\theta(W_{k,p,q}^0)^{(1-\theta)}J_{k,p,q}^1+(J_{k,1,q_1}^0)^\theta(J_{k,p,q}^0)^{(1-\theta)}W_{k,p,q}^1)\vspace{1ex}\\
&\qquad\qquad+B_{12}((W_{k,1,q_1}^0)^\theta(W_{k,p,q}^0)^{1-\theta}J_{k,p,q}^0+(J_{k,1,q_1}^0)^\theta(J_{k,p,q}^0)^{1-\theta}W_{k,p,q}^0).
\end{array}\eqno(3.108)$$
Therefore, we get from (3.101), (3.108), (3.12), (3.29) and (3.77) that
$$\begin{array}{ll}
&\quad W_{k+1,1,q_1}^1+J_{k+1,1,q_1}^1\vspace{1ex}\\
&\leq (B_7+B_{12})(\|\omega_0\|_{1,\tilde{q_0}}+\|j_0\|_{1,\tilde{q_0}})\vspace{1ex}\\
&\quad+B_7(W_{k,1,q_1}^0)^\theta(W_{k,p,q}^0)^{1-\theta}W_{k,p,q}^1+B_7(W_{k,1,q_1}^1)^\theta(W_{k,p,q}^1)^{1-\theta}W_{k,p,q}^0\vspace{1ex}\\
&\quad+B_7(J_{k,1,q_1}^1)^\theta(W_{k,p,q}^1)^{1-\theta}W_{k,p,q}^0+B_7(J_{k,1,q_1}^0)^\theta(J_{k,p,q}^0)^{1-\theta}J_{k,p,q}^1\vspace{1ex}\\
&\quad+B_7(W_{k,1,q_1}^0)^\theta(W_{k,p,q}^0)^{2-\theta}+B_7(J_{k,1,q_1}^0)^\theta(J_{k,p,q}^0)^{2-\theta}\vspace{1ex}\\
&\quad+B_{12}(W_{k,1,q_1}^0)^\theta(W_{k,p,q}^0)^{(1-\theta)}J_{k,p,q}^1
+B_{12}(J_{k,1,q_1}^0)^\theta(J_{k,p,q}^0)^{(1-\theta)}W_{k,p,q}^1\vspace{1ex}\\
&\quad+B_{12}W_{k,p,q}^0(J_{k,1,q_1}^1)^{\theta}(J_{k,p,q}^1)^{1-\theta}\vspace{1ex}\\
&\quad+B_{12}((W_{k,1,q_1}^0)^\theta(W_{k,p,q}^0)^{1-\theta}J_{k,p,q}^0+(J_{k,1,q_1}^0)^\theta(J_{k,p,q}^0)^{1-\theta}W_{k,p,q}^0)\vspace{1ex}\\
&\leq B_{13}(\|\omega_0\|_{1,\tilde{q_0}}+\|j_0\|_{1,\tilde{q_0}})+B_{13}(\|\omega_0\|_{p_0,q_0}+\|j_0\|_{p_0,q_0})^{2-\theta}\vspace{1ex}\\
&\quad+B_{13}(\|\omega_0\|_{p_0,q_0}+\|j_0\|_{p_0,q_0})^{2-\theta}(W_{k,1,q_1}^1+J_{k,1,q_1}^1)^\theta
\end{array}\eqno(3.109)$$
when (3.13), (3.30) and (3.76) hold. On the other hand, by Proposition 4,
it holds that
$$\begin{array}{ll}
&W_{1,1,q_1}^1+J_{1,1,q_1}^1\leq\sup\limits_{t\in\mathbb{R}^{+}}t^{\frac{1+q_1-\tilde{q_0}}{2}}\|\nabla G(\cdot,t)\ast\omega_0\|_{1,q_1}
+\sup\limits_{t\in\mathbb{R}^{+}}t^{\frac{1+q_1-\tilde{q_0}}{2}}\|\nabla G(\cdot,t)\ast j_0\|_{1,q_1}\vspace{1ex}\\
&\qquad\qquad\qquad\leq B_{14}(\|\omega_0\|_{1,\tilde{q_0}}+\|j_0\|_{1,\tilde{q_0}}).
\end{array}\eqno(3.110)$$

By (3.109), (3.110) and the arguments similar to those used to derive
(3.77) and (3.78), it holds that
$$W_{k,1,q_1}^1+J_{k,1,q_1}^1\leq 2\ \ \ {\rm for\ all}\ k\geq1,\eqno(3.111)$$
$$W_{k,1,q_1}^1+J_{k,1,q_1}^1\leq C_{12}(\|\omega_0\|_{1,\tilde{q_0}}+\|j_0\|_{1,\tilde{q_0}}+\|\omega_0\|_{p_0,q_0}+\|j_0\|_{p_0,q_0})\ \ \ {\rm for\ all}\ k\geq1,\eqno(3.112)$$
when (3.13), (3.30) and (3.76) hold and
$$(B_{13}+B_{14})(\|\omega_0\|_{1,\tilde{q_0}}+\|j_0\|_{1,\tilde{q_0}}+(\|\omega_0\|_{p_0,q_0}+\|j_0\|_{p_0,q_0})^{2-\theta})\leq1/2.\eqno(3.113)$$

{\bf Step 4: Estimates for  the terms $\bar{W}_{k,1,q_1}^1$,
$\bar{J}_{k,1,q_1}^1$.}

\medskip

By (3.3), Proposition 4 and (3.90), it holds that
$$\begin{array}{ll}
&\quad\|\nabla\omega^{(k+1)}\|_{1,q_1}\vspace{1ex}\\
&\leq\displaystyle\|\nabla G(\cdot,t)\ast \omega_0\|_{1,q_1}\vspace{1ex}\\
&\quad+\displaystyle B_4\int_0^t(t-s)^{-\frac{1}{2}-\frac{q_1-q_2}{2}}(\|\omega^{(k)}(\cdot,s)\|_{1,q_1}^\theta\|\omega^{(k)}(\cdot,s)\|_{p,q}^{(1-\theta)}\|\nabla\omega^{(k)}(\cdot,s)\|_{p,q}){\rm d}s\\
&\quad+\displaystyle B_4\int_0^t(t-s)^{-\frac{1}{2}-\frac{q_1-q_2}{2}}(\|\nabla\omega^{(k)}(\cdot,s)\|_{1,q_1}^\theta\|\nabla\omega^{(k)}(\cdot,s)\|_{p,q}^{(1-\theta)}\|\omega^{(k)}(\cdot,s)\|_{p,q}){\rm d}s\\
&\quad+\displaystyle B_4\int_0^t(t-s)^{-\frac{1}{2}-\frac{q_1-q_2}{2}}(\|\nabla j^{(k)}(\cdot,s)\|_{1,q_1}^\theta\|\nabla j^{(k)}(\cdot,s)\|_{p,q}^{(1-\theta)}\|j^{(k)}(\cdot,s)\|_{p,q}){\rm d}s\\
&\quad+\displaystyle B_4\int_0^t(t-s)^{-\frac{1}{2}-\frac{q_1-q_2}{2}}(\|j^{(k)}(\cdot,s)\|_{1,q_1}^\theta\|j^{(k)}(\cdot,s)\|_{p,q}^{(1-\theta)}\|\nabla j^{(k)}(\cdot,s)\|_{p,q}){\rm d}s.
\end{array}\eqno(3.114)$$
It was observed that
\begin{equation*}
\begin{array}{ll}
&\displaystyle\frac{1+q_1-\tilde{q_0}}{2}+1=\frac{1+q_1-q_2}{2}+\Big(\frac{q_2-\tilde{q_0}}{2}+1\Big),\vspace{1ex}\\
&\displaystyle\frac{2}{q_1-\tilde{q_0}}>\theta,\ \ \frac{2p}{2p-3+q}>1-\theta,\ \ \ \frac{2p}{3p-3+q}>1,\vspace{1ex}\\
&\displaystyle\frac{q_2-\tilde{q_0}}{2}+1=\frac{q_1-\tilde{q_0}}{2}\theta+\frac{2p-3+q}{2p}(1-\theta)+\frac{3p-3+q}{2p},\vspace{1ex}\\
&\displaystyle\frac{2}{1+q_1-\tilde{q_0}}>\theta,\ \ \ \frac{2p}{3p-3+q}>1-\theta,\ \ \ \frac{2p}{2p-3+q}>1,\vspace{1ex}\\
&\displaystyle\frac{q_2-\tilde{q_0}}{2}+1=\frac{1+q_1-\tilde{q_0}}{2}\theta+\frac{3p-3+q}{2p}(1-\theta)+\frac{2p-3+q}{2p}.
\end{array}
\end{equation*}
These facts together with (3.114), Proposition 4, Lemma 2, Young's
inequality and H\"{o}lder's inequality imply that
$$\begin{array}{ll}
&\bar{W}_{k+1,1,q_1}^1\leq B_{15}\|\omega_0\|_{1,\tilde{q_0}}+B_{15}\|t^{-\frac{1}{2}-\frac{q_1-q_2}{2}}\|_{L^{\frac{2}{q_1-q_2+1},\infty}(\mathbb{R}^{+})}\vspace{1ex}\\
&\qquad\qquad\quad\times\Big(\|(\|\omega^{(k)}(\cdot,t)\|_{1,q_1}^\theta\|\omega^{(k)}(\cdot,t)\|_{p,q}^{(1-\theta)}\|\nabla\omega^{(k)}(\cdot,t)\|_{p,q})\|_{L^{\frac{q_2-\tilde{q_0}+2}{2}}(\mathbb{R}^{+})}\vspace{1ex}\\
&\qquad\qquad\quad+\|\|\nabla\omega^{(k)}(\cdot,t)\|_{1,q_1}^\theta\|\nabla\omega^{(k)}(\cdot,t)\|_{p,q}^{(1-\theta)}\|\omega^{(k)}(\cdot,t)\|_{p,q}\|_{L^{\frac{q_2-\tilde{q_0}+2}{2}}(\mathbb{R}^{+})}\vspace{1ex}\\
&\qquad\qquad\quad+\|\|\nabla j^{(k)}(\cdot,t)\|_{1,q_1}^\theta\|\nabla j^{(k)}(\cdot,t)\|_{p,q}^{(1-\theta)}\|j^{(k)}(\cdot,t)\|_{p,q}\|_{L^{\frac{q_2-\tilde{q_0}+2}{2}}(\mathbb{R}^{+})}\vspace{1ex}\\
&\qquad\qquad\quad+\|\|j^{(k)}(\cdot,t)\|_{1,q_1}^\theta\|j^{(k)}(\cdot,t)\|_{p,q}^{(1-\theta)}\|\nabla j^{(k)}(\cdot,t)\|_{p,q}\|_{L^{\frac{q_2-\tilde{q_0}+2}{2}}}\Big)\vspace{1ex}\\
&\qquad\qquad\leq B_{15}\|\omega_0\|_{1,\tilde{q_0}}\vspace{1ex}\\
&\qquad\qquad\quad+B_{15}((\bar{W}_{k,1,q_1}^0)^\theta(\bar{W}_{k,p,q}^0)^{1-\theta}\bar{W}_{k,p,q}^1+(\bar{J}_{k,1,q_1}^0)^\theta(\bar{J}_{k,p,q}^0)^{1-\theta}\bar{J}_{k,p,q}^1)\vspace{1ex}\\
&\qquad\qquad\quad+B_{15}((\bar{W}_{k,1,q_1}^1)^\theta(\bar{W}_{k,p,q}^1)^{1-\theta}\bar{W}_{k,p,q}^0+(\bar{J}_{k,1,q_1}^1)^\theta(\bar{J}_{k,p,q}^1)^{1-\theta}\bar{J}_{k,p,q}^0).
\end{array}\eqno(3.115)$$
By (3.3), Proposition 4 and (3.91), it holds that
$$\begin{array}{ll}
&\quad\|\nabla j^{(k+1)}\|_{1,q_1}\vspace{1ex}\\
&\leq\|\nabla G(\cdot,t)\ast j_0\|_{1,q_1}\vspace{1ex}\\
&\quad+\displaystyle B_4\int_0^t(t-s)^{-\frac{1}{2}-\frac{q_1-q_2}{2}}\|\omega^{(k)}(\cdot,s)\|_{1,q_1}^\theta\|\omega^{(k)}(\cdot,s)\|_{p,q}^{(1-\theta)}\|\nabla j^{(k)}(\cdot,s)\|_{p,q}{\rm d}s\\
&\quad+\displaystyle 2B_4\int_0^t(t-s)^{-\frac{1}{2}-\frac{q_1-q_2}{2}}\|\omega^{(k)}(\cdot,s)\|_{p,q}\|\nabla j^{(k)}(\cdot,s)\|_{1,q_1}^\theta\|\nabla j^{(k)}(\cdot,s)\|_{p,q}^{(1-\theta)}{\rm d}s\\
&\quad+\displaystyle B_4\int_0^t(t-s)^{-\frac{1}{2}-\frac{q_1-q_2}{2}}\|j^{(k)}(\cdot,s)\|_{1,q_1}^\theta\|j^{(k)}(\cdot,s)\|_{p,q}^{(1-\theta)}\|\nabla\omega^{(k)}(\cdot,s)\|_{p,q}{\rm d}s.
\end{array}\eqno(3.116)$$
Inequality (3.116) together with the arguments similar to those 
used to derive (3.115) yields that
$$\begin{array}{ll}
&\bar{J}_{k+1,1,q_1}^1\leq B_{16}\|j_0\|_{1,\tilde{q_0}}+B_{16}\bar{W}_{k,p,q}^0(\bar{J}_{k,1,q_1}^1)^\theta(\bar{J}_{k,p,q}^1)^{1-\theta}\vspace{2ex}\\
&\qquad\qquad+B_{16}((\bar{W}_{k,1,q_1}^0)^\theta(\bar{W}_{k,p,q}^0)^{1-\theta}\bar{J}_{k,p,q}^1+(\bar{J}_{k,1,q_1}^0)^\theta(\bar{J}_{k,p,q}^0)^{1-\theta}\bar{W}_{k,p,q}^1).
\end{array}\eqno(3.117)$$
Hence, by (3.12), (3.29), (3.39), (3.85), (3.115) and (3.117), we have
$$\begin{array}{ll}
&\bar{W}_{k+1,1,q_1}^1+\bar{J}_{k+1,1,q_1}^1\leq B_{17}(\|\omega_0\|_{1,\tilde{q_0}}+\|j_0\|_{1,\tilde{q_0}})+B_{17}(\|\omega_0\|_{p_0,q_0}+\|j_0\|_{p_0,q_0})^{2-\theta}\vspace{2ex}\\
&\qquad\qquad\qquad\qquad\quad+B_{17}(\|\omega_0\|_{p_0,q_0}+\|j_0\|_{p_0,q_0})^{2-\theta}(\bar{W}_{k,1,q_1}^1+\bar{J}_{k,1,q_1}^1)^\theta
\end{array}\eqno(3.118)$$
whenever (3.13), (3.30), (3.40) and (3.86) hold. On the other hand,
applying Proposition 4 and Lemma 2, we can get
$$\begin{array}{ll}
&\quad\bar{W}_{1,1,q_1}^1+\bar{J}_{1,1,q_1}^1\vspace{1ex}\\
&\leq\|\|\nabla G(\cdot,t)\ast\omega_0\|_{1,q_1}\|_{L^{\frac{2}{1+q_1-\tilde{q_0}}}(\mathbb{R}^{+})}+\|\|\nabla G(\cdot,t)\ast j_0\|_{1,q_1}\|_{L^{\frac{2}{1+q_1-\tilde{q_0}}}(\mathbb{R}^{+})}\vspace{1ex}\\
&\leq B_{18}(\|\omega_0\|_{1,\tilde{q_0}}+\|j_0\|_{1,\tilde{q_0}}).
\end{array}\eqno(3.119)$$
Inequality (3.118) together with (3.119) and similar arguments to 
those in getting (3.77) and (3.78) leads to
$$\bar{W}_{k,1,q_1}^1+\bar{J}_{k,1,q_1}^1\leq 2\ \ \ {\rm for\ all}\ k\geq1,\eqno(3.120)$$
$$\bar{W}_{k,1,q_1}^1+\bar{J}_{k,1,q_1}^1\leq C_{13}(\|\omega_0\|_{1,\tilde{q_0}}+\|j_0\|_{1,\tilde{q_0}}+\|\omega_0\|_{p_0,q_0}+\|j_0\|_{p_0,q_0})\ \ \ {\rm for\ all}\ k\geq1,\eqno(3.121)$$
when (3.13), (3.30), (3.40) and (3.86) hold and
$$(B_{17}+B_{18})(\|\omega_0\|_{1,\tilde{q_0}}+\|j_0\|_{1,\tilde{q_0}})+B_{17}(\|\omega_0\|_{p_0,q_0}+\|j_0\|_{p_0,q_0})^{2-\theta}\leq1/2.$$

The rest of proof is essentially analogous to {\bf Steps 4} and {\bf 5}
in the proof of Theorem \ref{thm3.1} (i). We omit the details. $\hfill\Box$

We now turn to prove Theorem \ref{thm2.2}.

\begin{proof}[Proof of Theorem \ref{thm2.2}] Let $p_0=q_0=1$, then
$E_1=A_1$. This proves (i) of Theorem \ref{thm2.2}. Taking
$q_3=q_1$ and $\tilde{q_0}=1$ in $E_2$. One can easily get that
$$q_1=q_2=q_3=q,\ \ \ \tilde{p}=\frac{p'(3-q)}{3-q+p'}.$$
By the fact that $1<\tilde{p}<\min\{p,p'\}$ and $0\leq q_1-\tilde{q_0}<1$,
it holds that
$$q\in[1,2),\ \ \ \frac{2(3-q)}{4-q}<p<3-q.$$
This proves (ii) of Theorem \ref{thm2.2}.
\end{proof}

\bigskip

\quad\hspace{-20pt}{\bf Appendix}

\medskip

This appendix will be devoted to presenting some notations, lemmas
and propositions, which are frequently used in  our proof.

\bigskip

\quad\hspace{-20pt}{\bf Appendix A}

\medskip
Appendix A contains some propositions. We start with some basic
properties of Morrey spaces:

\medskip
\quad\hspace{-20pt}{\bf Proposition 1 (Basic properties of Morrey
space)}{\rm :}\quad{\it
\begin{itemize}
\item [(i)] $\mathcal{M}^{1,0}(\mathbb{R}^3)=\mathcal{M}^{1}
(\mathbb{R}^3)$ is the set of finite measures $\mathcal{M}$
and $\|\mu\|_1=|\mu|$;

\medskip
\item [(ii)] $\mathcal{M}^{p,0}(\mathbb{R}^3)=L^p(\mathbb{R}^3)$
for $p>1$;

\medskip
\item [(iii)] $L^p(\mathbb{R}^3)\subset L^{p,\infty}(\mathbb{R}^3)
\subset\mathcal{M}^p(\mathbb{R}^3)$ for $1<p<\infty$, where
$L^{p,\infty}(\mathbb{R}^3)$ denotes the Lorentz space;

\medskip
\item [(iv)] Inclusion relations: for $1\leq r, \ s,\ \tau, \
\lambda<\infty$ satisfying  $s\leq r$, $\tau\leq\lambda$ and
$\frac{3-\lambda}{r}=\frac{3-\tau}{s}$,
\begin{equation*}
\mathcal{M}^{r,\lambda}(\mathbb{R}^3)\subset\mathcal{M}^{s,\tau}(\mathbb{R}^3);
\end{equation*}

\medskip
\item [(v)] Interpolation inequality: if $\frac{1}{p}=\frac{1-\theta}{p_0}
+\frac{\theta}{p_1}$ and $0<\theta<1$, then
$$\mathcal{M}^{p_0}(\mathbb{R}^3)\cap
\mathcal{M}^{p_1}(\mathbb{R}^3)\subset\mathcal{M}^p(\mathbb{R}^3)$$
and
\begin{equation*}
\|\mu\|_p\leq\|\mu\|_{p_0}^{1-\theta}\|\mu\|_{p_1}^{\theta}\mbox{\ \  for \ } \mu \in\mathcal{M}^{p_0}(\mathbb{R}^3)\cap\mathcal{M}^{p_1}(\mathbb{R}^3).
\end{equation*}

\medskip
\item[(vi)] Let $1\leq p_1<p_3<p_2$,\ $0\leq\mu_1,\,\mu_2,\,\mu_3<3$
and $k\in(0,1)$ be such that
$$\frac{1}{p_3}=\frac{k}{p_1}+\frac{1-k}{p_2},\ \ \ \frac{\mu_3}{p_3}=\frac{\mu_1}{p_1}k+\frac{\mu_2}{p_2}(1-k).$$
Then
$$\|f\|_{p_3,\mu_3}\leq C\|f\|_{p_1,\mu_1}^k\|f\|_{p_2,\mu_2}^{1-k}.$$
\end{itemize}}

\begin{proof} It should be pointed out that (i)-(iii) and (v) of
Proposition 1 follows from \cite{GM2} and (iv) of Proposition 1
follows from \cite{Ka2}. The proof of (vi) of Proposition 1 follows
easily from the arguments same to those used to derive
\cite[Lemma 2.1 (iv)]{AF}. Here we omit the details.
\end{proof}

\medskip

\quad\hspace{-20pt}{\bf Remark 1.} {By Proposition 1 (iii), some
works of \cite{GM2,Ka2} can be regarded as the  generalizations
of the $L^p$ theory on NS problem in \cite{Gi,GM,Ka1,Wahl}.}

\medskip

\quad\hspace{-20pt}{\bf Propostion 2.}\quad
{\it Let $\mu=(\mu_1,\mu_2,\mu_3)$, $\nu=(\nu_1,\nu_2,\nu_3)$ and
$\omega=\nabla\times u$, then one has
\begin{itemize}
\item[(i)]{\bf (The H\"{o}lder inequality in Morrey space)}{\rm :} for
$1\leq r,\ s, \ m, \ \tau\leq\infty$ satisfying  $\frac{1}{r}=\frac{1}{m}
+\frac{1}{s}$ and $\frac{\theta}{r}=\frac{\lambda}{m}+\frac{\tau}{s}$, then
\begin{equation*}
\|\mu\nu\|_{r,\theta}\leq\|\mu\|_{m,\lambda}\|\nu\|_{s,\tau}.
\end{equation*}
Particularly, if $1\leq p\leq\infty$, $\frac{1}{\theta}
+\frac{1}{r}=1$ and $\frac{3}{p'}=\frac{\tau}{\theta}+\frac{s}{r}$,
then
 \begin{equation*}
 \begin{array}{ll}
&\|\mu\nu\|_p\leq\|\mu\|_p\|\nu\|_\infty,\vspace{1ex}\\
&\|\nabla u b\|_{p}\leq\|\omega\|_{\theta,\tau}\|b\|_{r,s}.
\end{array}
\end{equation*}
\item[(ii)] {\bf (Inequalities about the Biot-Savart kernel)}{\rm :}
\begin{itemize}
\item [(a)]If $\frac{1}{p}=\frac{1}{q}+\frac{1}{3}$ and
$\mu \in{\mathcal{M}}^{p}(\mathbb{R}^3)$, then
$K\ast\mu\in{\mathcal{M}}^{q}(\mathbb{R}^3)$ and
\begin{equation*}
\|K\ast \mu\|_q\leq\frac{1}{4\pi}\|\mu\|_p.
\end{equation*}
\item [(b)]If $0\neq p<3<q$ and $\mu\in{\mathcal{M}}^{p}
(\mathbb{R}^3)\cap{\mathcal{M}}^{q}(\mathbb{R}^3)$, then
$K\ast\mu\in L^{\infty}(\mathbb{R}^3)$ and
\begin{equation*}
\|K\ast \mu\|_\infty\leq\frac{1}{4\pi}\|\mu\|_p^{\big(\frac{1}{q}-\frac{1}{p}\big)^{-1}\big(\frac{2}{3}-\frac{1}{q'}\big)}\|\mu\|_q^{\big(\frac{1}{q}-\frac{1}{p}\big)^{-1}\big(\frac{1}{p'}-\frac{2}{3}\big)}.
\end{equation*}
Particularly, if we choose $q=2p$ and $\theta=\frac{2p}{3}$ with
$p\in(\frac{3}{2},3)$, then
\begin{equation*}
\|K\ast \mu\|_\infty\leq\|\mu\|_p^{\theta-1}\|\mu\|_{2p}^{2-\theta}.
\end{equation*}
\end{itemize}
\end{itemize}}

\medskip

The following proposition focuses the mapping properties for the
Riesz potential on the Morrey spaces (see \cite[Proposition 3.7]{Tay}).

\medskip

\quad\hspace{-20pt}{\bf Proposition 3.}\quad{\it Let $S(x)=|x|^{\delta-3}$
for some $\delta\in(0,3)$. Then, for $1<p<q<\infty$, $0\leq\theta<3$,
$\frac{1}{p}-\frac{1}{q}=\frac{\delta}{3-\theta}$ and
$f\in\mathcal{M}^{p,\theta}(\mathbb{R}^n)$, there exists a positive
constant $C$ independent of $f$ such that
$$\|S\ast f\|_{q,\theta}\leq C\|f\|_{p,\theta}.$$}

\medskip

By motivated by the idea in \cite{Ka2,Pe}, we have the following
inequalities for the heat kernel and Biot-Savart kernel:

\medskip

\quad\hspace{-20pt}{\bf Proposition 4.}\quad{\it
Let $1\leq q_1\leq q_2\leq\infty$, \  $0\leq\lambda_1\leq\lambda_2<3$,\ and
for $t>0$,
$$G(x,t)=(4\pi t)^{-\frac{3}{2}}\exp(-\frac{|x|^2}{4 t}).$$
We define the following operators as follows:
\begin{equation*}
T_{1,t}f=G(\cdot,t)*f(x),\ \ T_{2,t}f=\nabla G(\cdot,t)*f(x),\ \ T_{3,t}f=\partial_tG(\cdot,t)*f(x).
\end{equation*}
Then the operators $T_{i,t}\,(i=1,2,3)$ are bounded  from
$\mathcal{M}^{q_1,\lambda_1}(\mathbb{R}^3)$ to $\mathcal{M}^{q_2,
\lambda_2}(\mathbb{R}^3)$ and depend on $t$ continuously.
Furthermore, one has for $f\in \mathcal{M}^{q_1,\lambda_1}(\mathbb{R}^3)$,
$$t^{\frac{1}{2}(\alpha_1-\alpha_2)}\|T_{1,t}f\|_{q_2,\lambda_2}\leq C\|f\|_{q_1,\lambda_1},\eqno(A.1)$$
$$t^{\frac{1}{2}+\frac{1}{2}(\alpha_1-\alpha_2)}\|T_{2,t}f\|_{q_2,\lambda_2}\leq C\|f\|_{q_1,\lambda_1},\eqno(A.2)$$
$$t^{1+\frac{1}{2}(\alpha_1-\alpha_2)}\|T_{3,t}f\|_{q_2,\lambda_2}\leq C\|f\|_{q_1,\lambda_1},\eqno(A.3)$$
where $\alpha_i=\frac{3-\lambda_i}{q_i}\,(i=1,2)$ and constants $C$
depends on $q_1,\,q_2,\,\lambda_1,\,\lambda_2$.}

\begin{proof}
The case $\lambda_1=\lambda_2$ was proved in \cite{Wu}. We shall
adopt the ideas used in the proof of Proposition 2.4 in \cite{Wu}
to prove the case $\lambda_1<\lambda_2$. It was shown in \cite{Wu}
that
\begin{equation*}
\partial_xG(x,t)=ct^{-\frac{1}{2}}g_t(x),\ \ \ \partial_tG(x,t)=ct^{-1}g_t(x).
\end{equation*}
Moreover, the function $g_t$ is another radial function enjoying
the same properties as $G(x,t)$ does. Hence, we only prove (A.1)
since (A.2) and (A.3) can be proved similarly.

We now prove (A.1). Let $\beta=\frac{n-\lambda_1}{n-\lambda_2}$.
It is clear that $\beta>1$ since $\lambda_1<\lambda_2<n$. Fix
$t>0$ and $R>0$. It is clear that
$$\int_{|x-y|<R}|T_{1,t}f(y)|^{q_2}{\rm d}y\leq\|T_{1,t}f\|_{L^\infty}^{q_2-\frac{q_1}{\beta}}\int_{|x-y|<R}|T_{1,t}f(y)|^{\frac{q_1}{\beta}}{\rm d}y.\eqno(A.4)$$
By H\"{o}lder's inequality with exponents $p=\beta$ and $p'=\beta'$,
we see that
$$\int_{|x-y|<R}|T_{1,t}f(y)|^{\frac{q_1}{\beta}}{\rm d}y\leq CR^{n(1-\frac{1}{\beta})}\Big(\int_{|x-y|<R}|T_{1,t}f(y)|^{q_1}{\rm d}y\Big)^{{1}/{\beta}}.\eqno(A.5)$$
(A.5) together with (A.4) yields that
$$R^{-\lambda_2}\int_{|x-y|<R}|T_{1,t}f(y)|^{q_2}{\rm d}y\leq C\|T_{1,t}f\|_{L^\infty}^{q_2-\frac{q_1}{\beta}}\|T_{1,t}f\|_{q_1,\lambda_1}^{\frac{q_1}{\beta}}.\eqno(A.6)$$
Combining (A.6) with the known estimates (A.2) and (A.3) in
\cite{Wu} implies that
$$R^{-\lambda_2}\int_{|x-y|<R}|T_{1,t}f(y)|^{q_2}{\rm d}y\leq C\|f\|_{q_1,\lambda_1}^{q_2}t^{\frac{\lambda_1-n}{2q_1}(q_2-\frac{q_1}{\beta})}=C\|f\|_{q_1,\lambda_1}^{q_2}t^{\frac{q_2}{2}(\alpha_2-\alpha_1)},\eqno(A.7)$$
where $\alpha_i=\frac{3-\lambda_i}{q_i}\,(i=1,2)$. Then (A.1)
follows easily from (A.7).
\end{proof}

\bigskip

\quad\hspace{-20pt}{\bf Appendix B}

\medskip
Appendix B is devote to presenting some technique lemmas, which
are useful in the proof of of Theorem \ref{thm3.1}. Let $\mathcal{B}(a,b)$
be the beta function defined by
\begin{equation*}
\mathcal{B}(a,b)=\int_0^t(t-s)^{a-1}s^{b-1}{\rm d}s.
\end{equation*}
It is well-known that if $a,b>0$, then
$$\mathcal{B}(a,b)=\mathcal{C}(a,b)t^{a+b-1}\ \ \ {\rm with}\ \ \mathcal{C}(a,b)=\int_0^1(1-s)^{a-1}s^{b-1}{\rm d}s>0.$$

\medskip
\quad\hspace{-20pt}{\bf Lemma 1.}\quad{\it
Let $\{x_k\}_{k\geq0}$ be a sequence of nonnegative real numbers
and $f:[0,\infty)\rightarrow\mathbb{R}$ be a function. Suppose
that $f$ satisfies the following conditions:
\begin{itemize}
\item[{\rm (i)}] $f(x)=x$ have a positive root $x^{*}$;

\medskip

\item[{\rm (ii)}] $f(x)$ is monotonically non-decreasing in $[0,x^{*}]$;

\medskip

\item[{\rm (iii)}] $x_{k+1}\leq f(x_k)$ for all $k\geq0$.

\end{itemize}

If $x_0\leq x^{*}$, then $x_k\leq x^{*}$ for all $k\geq0$.}

\begin{proof}
It is clear that $x_0\leq x^{*}$. Assume that $x_{l-1}\leq x^{*}$
for all $l\in\{0,1,2,\ldots,k\}$ with some $k\geq1$. This assumption
yields that
$$x_k-x^{*}\leq f(x_{k-1})-x^{*}\leq f(x^{*})-x^{*}=0,$$
which yields $x_k\leq x^{*}$. This concludes the desired conclusion
by induction.
\end{proof}

As a direct application of Lemma 1, we can get the following
result.

\medskip

\quad\hspace{-20pt}{\bf Corollary 1.}\quad{\it
Let $a_1>0$, $b_1>0$ and $1-4a_1b_1>0$. Let $\{X_k\}_{k\geq0}$
be a sequence of nonnegative real numbers such that

\medskip
\begin{itemize}
\item[{\rm (i)}] $X_0\leq\frac{1-\sqrt{1-4a_1b_1}}{2b_1}$;

\medskip

\item[{\rm (ii)}] $X_{k+1}\leq a_1+b_1X_k^2$ for all $k\geq0$.

\end{itemize}
Then one has
\begin{equation*}
X_k\leq\frac{1-\sqrt{1-4a_1b_1}}{2b_1}=\frac{2a_1}{1+\sqrt{1-4a_1b_1}}<2a_1.
\end{equation*}}

The following results  play key roles in the proof of Theorem
\ref{thm3.1}.

\medskip

\quad\hspace{-20pt}{\bf Lemma 2.}\quad {\it
Let $G(x,t)=(4\pi t)^{-\frac{3}{2}}\exp(-\frac{|x|^2}{4 t})$
and $k=0,1$. We denote $\nabla^{0}u=u$ and $\nabla^1u=\nabla u$.
Suppose that for fixed $p_0\in[1,\infty)$, $q_0\in(0,3)$ and
$u\in\mathcal{M}^{p_0,q_0}(\mathbb{R}^n)$, there exists a constant
$A>0$ independent of $u,\,t$ such that
\begin{equation*}
\|\nabla^kG(\cdot,t)*u\|_{p,q}\leq At^{-\frac{1}{2}\big(\frac{3-q_0}{p_0}-\frac{3-q}{p}\big)-\frac{k}{2}}\|u\|_{p_0,q_0}
\end{equation*}
for all $p\in(p_0,\infty)$ and $q\in[q_0,3)$. Fix $p_0\in[1,\infty)$,
$q_0\in(0,3)$ and $u_0\in\mathcal{M}^{p_0,q_0}(\mathbb{R}^n)$, then
for any $p\in(p_0,\infty)$ and $q\in[q_0,3)$, there exists a constant
$C>0$ independent of $u_0$ such that
$$\|\|\nabla^kG(\cdot,t)*u_0\|_{p,q}\|_{L_t^{a}}\leq A\|u_0\|_{p_0,q_0},\eqno(A.8)$$
where $$a=2\Big(\frac{3-q_0}{p_0}-\frac{3-q}{p}+k\Big)^{-1}.$$}
\begin{proof}
Fix $p>p_0$, $q\geq q_0$ and $q\in(0,3)$. By our assumption we known
that there exist $p_1,p_2$ with $p_1<p$, $p_2<p$, $p_1<p_0<p_2$ such
that
$$\|\nabla^kG(\cdot,t)*u_0\|_{p,q}\leq At^{-\frac{1}{2}\big(\frac{3-q_0}{p_1}-\frac{3-q}{p}+k\big)}\|u_0\|_{p_1,q_0},\eqno(A.9)$$
$$\|\nabla^kG(\cdot,t)*u_0\|_{p,q}\leq At^{-\frac{1}{2}\big(\frac{3-q_0}{p_2}-\frac{3-q}{p}+k\big)}\|u_0\|_{p_2,q_0}.\eqno(A.10)$$
There exists a constant $\theta\in[0,1]$ such that $\frac{\theta}{p_1}
+\frac{1-\theta}{p_2}=\frac{1}{p_0}$. Let
\begin{equation*}
a_1=\frac{2}{\frac{3-q_0}{p_1}-\frac{3-q}{p}+k},\ \ \ a_2=\frac{2}{\frac{3-q_0}{p_2}-\frac{3-q}{p}+k}.
\end{equation*}
(A.9)-(A.10) together with the fact that
$\|t^{-a}\|_{L_t^{1/a,\infty}}=1$ yield that
$$\|\|G(\cdot,t)*u_0\|_{p,q}\|_{L_t^{a_1,\infty}}\leq A\|u_0\|_{p_1,q_0},\eqno(A.11)$$
$$\|\|G(\cdot,t)*u_0\|_{p,q}\|_{L_t^{a_2,\infty}}\leq A\|u_0\|_{p_2,q_0}.\eqno(A.12)$$
Notice that $\frac{\theta}{a_1}+\frac{1-\theta}{a_2}=\frac{1}{a}$. An
interpolation between (A.11) and (A.12) may yields (A.8).
\end{proof}

\medskip

\quad\hspace{-20pt}{\bf Lemma 3.}\quad{\it
Let $G(x,t)=(4\pi t)^{-\frac{3}{2}}\exp(-\frac{|x|^2}{4t})$. Assume
that given $p_0\in[1,\infty)$, $q_0\in(0,3)$ and
$u\in\mathcal{M}^{p_0,q_0}(\mathbb{R}^n)$, there exists a constant
$A>0$ independent of $u,\,t$ such that
\begin{equation*}
\|G(\cdot,t)*u\|_{p,q}\leq At^{-\frac{1}{2}\big(\frac{3-q_0}{p_0}-\frac{3-q}{p}\big)}\|u(\cdot,t)\|_{p_0,q_0}
\end{equation*}
for all $p\in(p_0,\infty)$ and $q\in[q_0,3)$. Fix $p_0\in[1,\infty)$,
$q_0\in(0,3)$, $p\in(p_0,\infty)$ and $q\in[q_0,3)$ such that
$$a=\frac{1}{2}\big(\frac{3-q_0}{p_0}-\frac{3-q}{p}\big)<\frac{1}{2}.$$
Let $u_0$ be a function such that
$$\|u_0\|_{p_0,q_0}\leq\frac{1}{4AB_1B_2}\min\Big\{\frac{1}{\mathcal{C}(a,1-2a)},
1\Big\},$$
and $\{u_k\}_{k\geq1}$ be a sequence of functions
satisfying the following
\begin{equation*}
\left\{\aligned
&u_1(x,t)=G(\cdot,t)*u_0(x),\\
&\|u_{k+1}(\cdot,t)\|_{p,q}\leq B_1\|G(\cdot,t)*u_0\|_{p,q}+B_2\int_0^t(t-s)^{a-1}\|u_k(\cdot,s)\|_{p,q}^2{\rm d}s
\endaligned\right.
\end{equation*}
for some $B_1,\,B_2>0$. Then, for all $k\geq1$, it holds that

\begin{itemize}
\item[{\rm (i)}] $\sup\limits_{t>0}t^a\|u_k\|_{p,q}\leq 2AB_1\|u_0\|_{p_0,q_0},$

\medskip

\item[{\rm (ii)}] $\|\|u_k\|_{p,q}\|_{L_t^{1/a}}\leq 2AB_1\|u_0\|_{p_0,q_0}$.
\end{itemize}}

\begin{proof}
Letting $A_k=\sup_{t>0}t^a\|u_k\|_{p,q}$. By our assumption, it holds that
\begin{equation*}
\begin{array}{ll}
&A_{k+1}\leq\displaystyle AB_1\|u_0\|_{p_0,q_0}+B_2t^a\int_0^t(t-s)^{a-1}\|u_k(s)\|_{p,q}^2{\rm d}s\vspace{1ex}\\
&\qquad\leq\displaystyle AB_1\|u_0\|_{p_0,q_0}+B_2\mathcal{C}(a,1-2a)\sup\limits_{t>0}t^{2a}\|u_k\|_{p,q}^2\vspace{1ex}\\
&\qquad\leq AB_1\|u_0\|_{p_0,q_0}+B_2\mathcal{C}(a,1-2a)A_k^2.
\end{array}
\end{equation*}
Notice that $1-4AB_1B_2\mathcal{C}(a,1-2a)\|u_0\|_{p_0,q_0}>1$. Invoking
Corollary 1, we can get
\begin{equation*}
A_k<2AB_1\|u_0\|_{p_0,q_0}.
\end{equation*}
This proves (i). Let $\tilde{A}_k=\|\|u_k\|_{p,q}\|_{L_t^{1/a}}$.
It follows from Lemma 2 and Young's inequality that
\begin{equation*}
\|f*g^2\|_{L_t^{1/a}}\leq\|f\|_{L_t^{{1}/{(1-a)},\infty}}\|g\|_{L_t^{1/a}}^2,
\end{equation*}
\begin{equation*}
\begin{array}{ll}
&\tilde{A}_{k+1}\leq B_1\|\|G(\cdot,t)*u_0\|_{p,q}\|_{L_t^{1/a}}+B_2\Big\|\int_0^t(t-s)^{a-1}\|u_k(s)\|_{p,q}^2{\rm d}s\Big\|_{L_t^{1/a}}\\
&\qquad\leq AB_1\|u_0\|_{p_0,q_0}+\|t^{a-1}\|_{L_t^{1/(1-a),\infty}}B_2\|\|u_k\|_{p,q}\|_{L_t^{1/a}}^2\\
&\qquad\leq AB_1\|u_0\|_{p_0,q_0}+B_2\tilde{A}_k^2,
\end{array}
\end{equation*}
which, along  with Corollary 1, yields (ii).
\end{proof}

Finally, we would like to remark that the following general formulas
are useful for calculations with vector fields in $\mathbb{R}^3$.
 \begin{equation*}
 \begin{array}{ll}
&\nabla(F\cdot G)=(F\cdot\nabla) G+(G\cdot\nabla) F+F\times(\nabla\times G)+G\times(\nabla\times F),\vspace{1ex}\\
&{\rm div}(F\times G)=G\cdot(\nabla\times F)-F\cdot(\nabla\times G),\vspace{1ex}\\
&\nabla\times(F\times G)=F{\rm div}G-G{\rm div}F+(G\cdot\nabla)F-(F\cdot\nabla) G.
\end{array}
  \end{equation*}

\medskip

\quad\hspace{-20pt}{\bf Acknowledgements.} The research of Feng Liu was
supported partly by NNSF of China (grant No. 11701333). The research of
Shengguo Zhu was supported in part by the Royal Society--Newton
International Fellowships ( grant No. NF170015), and the Monash
University-Robert Bartnik Visiting Fellowships.

\medskip

{\bf Conflict of Interest:} The authors declare that they have no conflict of interest.

\end{document}